\documentclass{amsart}
\usepackage{amssymb}
\usepackage[utf8]{inputenc}
\usepackage{xcolor}
\usepackage{tikz}
\usepackage{tikz-cd}
\usetikzlibrary{decorations.pathmorphing,shapes,calc}
\usepackage{fullpage}
\usepackage{enumitem}
\usepackage{braket}
\usepackage{textcomp}
 
 \usepackage{mathtools}
 
\usepackage{hyperref}
\usepackage{cleveref}

\newtheorem{thm}{Theorem}[section]
\newtheorem{prop}[thm]{Proposition}

\newtheorem{lemma}[thm]{Lemma}
\newtheorem{cor}[thm]{Corollary}

\newtheorem{st}{Step}

\newtheorem{defthm}[thm]{Definition/Theorem}
\theoremstyle{definition}

\newtheorem{notation}[thm]{Notation}

\newtheorem{fact}[thm]{Fact}
\newtheorem{question}[thm]{Question}
\newtheorem{ex}[thm]{Example}
\theoremstyle{remark}

\newtheorem{remark}[thm]{Remark}
\newtheorem{rem}[thm]{Remark}

\newcommand{\ceil}[1]{\left\lceil#1\right\rceil}

\DeclareMathOperator{\HE}{HE}
\renewcommand{\st}{\mathrm{st}}
\DeclareMathOperator{\Crit}{Crit}

\DeclareMathOperator{\Aut}{Aut}

\DeclareMathOperator{\val}{val}

\DeclareMathOperator{\pr}{pr}
\DeclareMathOperator{\indeg}{indeg}
\DeclareMathOperator{\outdeg}{outdeg}
\newcommand{\Z}{\mathbb{Z}}
\newcommand{\C}{\mathbb{C}}

\newcommand{\Q}{\mathbb{Q}}
\newcommand{\A}{\mathbb{A}}
\newcommand{\calA}{\mathcal{A}}

\newcommand{\R}{\mathbb{R}}

\newcommand{\frakX}{\mathfrak{X}}
\newcommand{\boldu}{\mathbf{u}}
\newcommand{\boldx}{\mathbf{x}}
\newcommand{\boldz}{\mathbf{z}}
\renewcommand{\P}{\mathbb{P}}
\newcommand{\M}{\mathcal{M}}

\newcommand{\abs}[1]{\left\lvert#1\right\rvert}

\newcommand{\Mbar}{\overline{\mathcal{M}}}
\newcommand{\Mtilde}{\widetilde{\mathcal{M}}}
\newcommand{\into}{\hookrightarrow}
\newcommand{\deriv}[3][]{\frac{d^{#1} #2}{d{#3}^{#1}}}

\title{Stable curves and chromatic polynomials}
\author[Reinke]{Bernhard Reinke}
\address{Max Planck Institute for Mathematics in the Sciences, Inselstra{\ss}e 22, Leipzig 04103, Germany}
\email{Bernhard.Reinke@mis.mpg.de}

\author[Silversmith]{Rob Silversmith}
\address{Warwick Mathematics Institute, University of Warwick, Coventry CV4 7AL, UK}
\email{Rob.Silversmith@warwick.ac.uk}

\begin{document}

\begin{abstract}
    The intersection numbers of moduli spaces of stable curves $\Mbar_{g,m}$ are well-studied and are known to have rich combinatorial structure. We introduce a natural class of these intersection numbers $\omega_{G,g,m}$ indexed by finite simple graphs $G=(V,E)$. In genus zero, these numbers are closely related to several previously-studied quantities, including maximum likelihood degrees in algebraic statistics, counts of regions of certain hyperplane arrangements, and Kapranov degrees. We give two proofs of a simple closed formula $\omega_{G,g,m}=(-1)^{\abs{V}}\chi_G(-(2g-2+m)),$ where $\chi_G$ is the chromatic polynomial of $G$ --- one proof via intersection theory on moduli spaces of stable curves, and the other using the theory of hyperplane arrangements. We discuss several related questions and speculations, including new candidates for the chromatic polynomial of a directed graph.
\end{abstract}

\maketitle

\section{Introduction}
Let $\Mbar_{g,P}$ denote the moduli space of genus-$g$ stable curves with marked points indexed by a finite set $P$, which has dimension $3g-3+\abs{P}$. For any marked point $p\in P$, there is a corresponding \emph{$\psi$-class}, or \emph{cotangent class}, denoted $\psi_p\in H^2(\Mbar_{g,P})$ (see Section \ref{sec:MgnBasics}). The intersection theory of $\psi$-classes is well-studied, e.g. it is the subject of the Witten conjecture \cite{Witten1991}, proved by Kontsevich \cite{Kontsevich1992} (and later by others \cite{Mirzakhani2007,KazarianLando2007,OkounkovPandharipande2009}). Subsets $P'\subset P$ induce forgetful maps $\pi_{P'}:\Mbar_{g,P}\to\Mbar_{g,P'}$, which forget all marked points not in $P'$. 

Let $G=(V,E)$ be a (nonempty) finite simple graph. In this paper we define, and give a formula for, a natural class of intersection numbers on $\Mbar_{g,P}$ associated to $G$. These are defined by taking products of pullbacks of $\psi$-classes along forgetful maps, where the choices of forgetful maps reflect adjacency in $G$. We first introduce the definition and formula in the simplest special case, which is already interesting (Section \ref{sec:SpecialCase}). The general case and main theorem follow in Section \ref{sec:GeneralCase}.

\subsection{Main definition and theorem, special case}\label{sec:SpecialCase} For $v\in V$, let $N[v]\subseteq V$ denote the closed neighborhood of $v$ --- that is, $v$ and its neighbors. We define an intersection number on the moduli space of genus-zero stable curves, whose marking set is the vertex set $V$ together with three ``extra'' elements $a,b,c$, as follows:
\begin{align}\label{eq:IntersectionNumberFirstDef}
    \omega_G=\int_{\Mbar_{0,V\sqcup\{a,b,c\}}}\prod_{v\in V}\pi_{N[v]\sqcup\{a,b,c\}}^*(\psi_v).
\end{align}
It is natural to ask whether $\omega_G$ is equal to some previously-known graph invariant. Recall the \emph{chromatic polynomial} $\chi_G$, a fundamental object in algebraic graph theory, whose value at a nonnegative integer $x$ is the number of proper vertex colorings of $G$ with $x$ colors. The first case of our main theorem is the following simple-yet-surprising formula:
\begin{thm}[Main theorem, special case]\label{thm:Case03}
    $\omega_G=(-1)^{\abs{V}}\chi_G(-1).$
\end{thm}
\begin{ex}\label{ex:ExampleKD}
    We illustrate using the graph $G$ in Figure \ref{fig:ExampleOfMainTheorem}. The intersection number \eqref{eq:IntersectionNumberFirstDef} is equal to 
    \begin{align}\label{eq:ExampleIntersectionNumber}
        \omega_G=\int_{\Mbar_{0,1234abc}}\pi_{1234abc}^*(\psi_1)\cdot\pi_{123abc}^*(\psi_2)\cdot\pi_{123abc}^*(\psi_3)\cdot\pi_{14abc}^*(\psi_4)=12.
    \end{align}
    The value of \eqref{eq:ExampleIntersectionNumber} can be computed using the basic intersection-theoretic properties of $\Mbar_{0,n}$ (see Section \ref{sec:MgnBasics}), or using the implementations \texttt{admcycles} or \texttt{KapranovDegrees} in Sage \cite{admcycles,KapranovDegrees}. 
    \begin{figure}
    \centering
        $G=${\begin{tikzpicture}[baseline=-0.65ex]
        \draw (0,0) node {$\bullet$} node[above] {$1$};
        \draw (150:1) node {$\bullet$} node[left] {$2$};
        \draw (210:1) node {$\bullet$} node[left] {$3$};
        \draw (0:1) node {$\bullet$} node[right] {$4$};
        \draw (0,0)--(150:1)--(210:1)--(0,0)--(1,0);
    \end{tikzpicture}}
    \quad{\LARGE$\leadsto$}\quad$\begin{array}{l}
        N[1]=\{1,2,3,4\}  \\
        N[2]=\{1,2,3\}  \\
        N[3]=\{1,2,3\}  \\
        N[4]=\{1,4\}  \\
    \end{array}$\quad{\LARGE$\leadsto$}\quad$\left.\begin{array}{l}
        \pi_{1234abc}^*(\psi_1)  \\
        \pi_{123abc}^*(\psi_2)  \\
        \pi_{123abc}^*(\psi_3)  \\
        \pi_{14abc}^*(\psi_4)  \\
    \end{array}\quad\right\}$
    \begin{tabular}{c}
    multiply and integrate\\over $\Mbar_{1234abc}$
    \end{tabular}
    \caption{Illustrating how a graph yields an intersection number (special case)}
    \label{fig:ExampleOfMainTheorem}
\end{figure}
$G$ has chromatic polynomial $\chi_G(x)=x^4 - 4x^3 + 5x^2 - 2x,$ and indeed $(-1)^4\cdot\chi_G(-1)=12=\omega_G$, as expected from Theorem \ref{thm:Case03}.
\end{ex}
\begin{rem}
    Adding the extra marked points $a,b,c$ is natural for two reasons. First, $\Mbar_{0,V\sqcup\{a,b,c\}}$ has dimension $\abs{V}$, so the product in \eqref{eq:IntersectionNumberFirstDef} is in the correct degree to integrate. Second, $\Mbar_{0,N[v]\sqcup\{a,b,c\}}$ has dimension at least 1 by construction, so the pullbacks $\pi_{N[v]\sqcup\{a,b,c\}}^*(\psi_v)$ are well-defined and nonzero.
\end{rem}

\subsection{Main definition and theorem, general case}\label{sec:GeneralCase} Our main theorem generalizes Theorem \ref{thm:Case03} in two natural ways --- varying the genus of the stable curve, and varying the number of ``extra'' marked points (i.e. $a,b,c$ above). Fix nonnegative integers $g,m$ such that $2g-2+m>0$, and let $M$ be an arbitrary $m$-element set. We define a cohomology class 
\begin{align}\label{eq:PsiDef}
    \Psi_{G,g,m}=\prod_{v\in V}\pi_{N[v]\cup M}^*(\psi_v)\in H^*(\Mbar_{g,V\sqcup M}).
\end{align}
Note that $\Psi_{G,g,m}$ is a codimension-$\abs{V}$ class on a $(3g-3+\abs{V}+m)$-dimensional space; it is therefore natural to define an intersection \emph{number} by 
\begin{align}\label{eq:OmegaDef}
    \omega_{G,g,m}&=\int_{\Mbar_{g,V\sqcup M}}\Psi_{G,g,m}\cdot\pi_M^*([pt]),
\end{align}
where $[pt]\in H^*(\Mbar_{g,M})$ is the class of a point. 
We also define the special case\footnote{It is natural to define (by convention) $\pi_M^*([pt])=24$ when $(g,m)=(1,0)$, since for all $(g,m)$ with $2g-2+m > 0$ and for all $a\in M$ we have $\pi_M^*([pt])=24^g\cdot g!\cdot\pi_M^*(\psi_a^{3g-3+m})$.}:
\begin{align}\label{eq:Omega10}
    \omega_{G,1,0}=24\cdot\int_{\Mbar_{1,V}}\Psi_{G,1,0}.
\end{align}

Our main result states that the numbers $\omega_{G,g,m}$ recover, up to sign, the values of the chromatic polynomial $\chi_G$ at \emph{all negative integer inputs} (and, if $(g,m)=(1,0)$, the coefficient of the linear term of $\chi_G$):
\begin{thm}[Main theorem, general case]\label{thm:main}
    We have the formula
    \begin{align}\label{eq:MainEqn1}
        \omega_{G,g,m}=\begin{cases}
            (-1)^{\abs{V}}\chi_G(-(2g-2+m))&(g,m)\ne(1,0)\\
            (-1)^{\abs{V}-1}\left.\deriv{}{x}\chi_G(x)\right|_{x=0}&(g,m)=(1,0).
        \end{cases}
    \end{align}
\end{thm}

Note that $\omega_{G,0,3}=\omega_G$, and specializing Theorem \ref{thm:main} to $(g,m)=(0,3)$ recovers Theorem \ref{thm:Case03}. The exceptional case $(g,m)=(1,0)$ gives (up to sign) the coefficient of the linear term of $\chi_G$, which has several interpretations --- e.g. it is equal to $T(1,0)$, where $T(x,y)$ is the Tutte polynomial of $G$.

\subsection{Interpretation of the main theorem and its proofs} 
We give two proofs of Theorem \ref{thm:main}, which are substantially different from each other, and highlight different aspects of the geometry of the problem.

The first proof shows that both sides of \eqref{eq:MainEqn1} satisfy the same recursion and agree in initial cases, and hence must agree for all graphs $G$. This recursion, called the deletion-contraction formula, is a common tool in graph theory. Proving that (a) the initial cases agree and (b) the right side of \eqref{eq:MainEqn1} satisfies the deletion-contraction formula both require nontrivial geometric computations in the cohomology ring of $\Mbar_{g,n}$. We give a more detailed outline of the necessary tools and ideas in Section \ref{sec:ProofOutline}, and carry out the proof in Section \ref{sec:Proof}.

As is often the case for combinatorial identities, it is enlightening to find a ``combinatorial proof'' in addition to a recursive one. Our second proof provides an enumerative interpretation to the left side of \eqref{eq:MainEqn1}, at least when $g=0$. The roots of this interpretation are in the theory of hyperplane arrangements and maximum likelihood estimation --- $\omega_{G,0,m}$ counts critical points of a certain function $F$ defined on a hyperplane arrangement complement. Such critical points comprise the intersection of the vanishing loci of the logarithmic derivatives of $F$, and we establish the above enumerative interpretation by identifying a representative of each classes $\pi_{N[v]\sqcup M}^*(\psi_v)$ with the vanishing locus of one logarithmic derivative. (These loci lie in different ambient spaces, and the most technical part of the paper is making this identification rigorous.) To prove \eqref{eq:MainEqn1}, we then adapt classical theorems of Stanley, Zaslavsky, and Varchenko, which give interpretations of the right side of \eqref{eq:MainEqn1} in terms of hyperplanes. We outline this proof in detail in Section \ref{sec:MLDegrees}, and carry out the proof in Section \ref{sec:hyperplanes}.

Theorem \ref{thm:main} raises many interesting questions about the connection between algebraic graph theory and the tautological ring of the moduli space of curves. We discuss some of these in Section \ref{sec:FurtherQuestions}. We also give there a combinatorial application --- the construction \eqref{eq:OmegaDef} generalizes naturally to directed graphs, and using this, we define natural notions of the chromatic polynomial of a directed graph that do not appear in the literature. The existence of these polynomials raises many interesting questions, foremost among them whether they admit an enumerative interpretation. (See Question \ref{q:DigraphQuestions}.) Details about these constructions are in Section \ref{sec:Digraphs}.

\subsection{Outline of the first proof: Intersection theory on moduli spaces of curves}\label{sec:ProofOutline} 

The first proof proceeds by induction on the number of edges in $G$. The base case, proven in Section \ref{sec:BaseCase}, states that when $G$ has no edges --- and hence has chromatic polynomial $\chi_G(x)=x^{\abs{V}}$ --- we have
\begin{align}\label{eq:BaseCase}
    \omega_{G,g,m}=\begin{cases}(2g-2+m)^{\abs{V}}&(g,m,\abs{V})\ne(1,0,1)\\1&(g,m,\abs{V})=(1,0,1).\end{cases}
\end{align}
The key idea behind the proof of \eqref{eq:BaseCase} is to represent $[pt]\in H^*(\Mbar_{g,M})$ as the class of a zero-dimensional \emph{boundary stratum}, which allows us to express $\pi_M^*([pt])$ as an explicit sum of boundary strata.

For the inductive step (Section \ref{sec:InductiveStep}), we use the \emph{deletion-contraction recursion} familiar in graph theory. Let $e$ be an edge of $G$, let $G\setminus e$ denote the graph obtained by deleting $e$, and let $G/e$ denote the (simple) graph obtained by contracting $e$ (and removing loops and doubled edges). It is well-known, and elementary to show, that the chromatic polynomial satisfies, for any edge $e$ of $G$, the recursion 
\begin{align*}
    \chi_G=\chi_{G\setminus e}-\chi_{G/e},
\end{align*}
or equivalently
\begin{align}\label{eq:DeletionContractionChromatic}
    (-1)^{\abs{V}}\chi_G=(-1)^{\abs{V}}\chi_{G\setminus e}+(-1)^{\abs{V}-1}\chi_{G/e},
\end{align}
where we note that $\abs{V}-1$ is the number of vertices in $G/e.$ The main part of the proof establishes that the intersection numbers $\omega_{G,g,m}$ satisfy the \emph{same} recursion as $(-1)^{\abs{V}}\chi_G$, i.e. \begin{align}\label{eq:OmegaDeletionContraction}
    \omega_{G,g,m}=\omega_{G\setminus e,g,m}+\omega_{G/e,g,m},
\end{align}
which, together with the base case, implies Theorem \ref{thm:main}.

We now explain some of the geometry behind \eqref{eq:OmegaDeletionContraction}. There are three intersection numbers appearing in \eqref{eq:OmegaDeletionContraction}. Two of them (namely $\omega_{G,g,m}$ and $\omega_{G\setminus e,g,m}$) come from graphs on the same vertex set $V=V(G)=V(G\setminus e)$, and hence the $\omega_{G,g,m}$ and $\omega_{G\setminus e,g,m}$ are both defined as integrals are over the same moduli space $\Mbar_{g,V\sqcup M}$. The third (namely $\omega_{G/e,g,m}$) comes from a graph with one fewer vertex --- specifically, let $V'=V(G/e)=(V\setminus\{u,w\})\cup\{\star\}$, where $\star$ is the vertex to which $e$ was contracted. The intersection number $\omega_{G/e,g,m}$ is an integral over $\Mbar_{g,V'\sqcup M}$. Observe that we have a rather natural embedding $\iota:\Mbar_{g,V'\sqcup M}\into\Mbar_{g,V\sqcup M}$, which takes a $V'\sqcup M$-marked curve and glues on, at $\star$, a genus-zero bubble with markings $u$ and $w$. 
Pushing forward via this embedding gives us a way to compare cohomology classes on $\Mbar_{g,V\sqcup M}$ with classes on $\Mbar_{g,V'\sqcup M}$. Indeed, by a somewhat involved computation using geometric properties of $\psi$-classes, we show
\begin{align}\label{eq:ClassDeletionContraction}
    \Psi_{G,g,m}(G)\cdot\pi_M^*([pt])=\Psi_{G\setminus e,g,m}\cdot\pi_M^*([pt])+\iota_*(\Psi_{G/e,g,m}\cdot\pi_M^*([pt]))\in H^*(\Mbar_{g,V\sqcup M}).
\end{align}
The numerical identity \eqref{eq:OmegaDeletionContraction} follows from the cohomological identity \eqref{eq:ClassDeletionContraction} by integrating.

\begin{remark}
    The genus-zero intersection numbers $\omega_{G,0,m}$ are ``Kapranov degrees'' in the sense of \cite{BrakensiekEurLarsonLi2023}. Indeed, the recursion for Kapranov degrees in \cite[Thm. B]{BrakensiekEurLarsonLi2023} can be used to give an alternative, purely combinatorial, proof of \eqref{eq:OmegaDeletionContraction} when $g=0$.
\end{remark}

\subsection{Outline of the second proof: Oriented graphs,  hyperplane arrangements and maximum likelihood degrees}\label{sec:MLDegrees} In this section we explain how Theorem \ref{thm:main} fits naturally into an existing body of results in combinatorics and algebra, which is closely related to maximum likelihood estimation in algebraic statistics. From these connections, we ultimately obtain a second proof of Theorem \ref{thm:main}. We begin the discussion by recalling two classical characterizations of evaluations of $\chi_G$ at negative values.

\begin{thm}[{Stanley, \cite[Thm.~1.2]{Stanley1973}}]
\label{thm:Stanley}
For a simple graph $G = (V,E)$, the value $(-1)^{|V|} \chi_G(-1)$ is equal to the number of acyclic orientations of $G$. More generally, the value $(-1)^{|V|} \chi_G(-k)$ is equal to the number of pairs $(\sigma, O)$, where $\sigma : V \rightarrow \set{1, \dots, k}$ and $O$ is an acyclic orientation of $G$ subject to the condition that if $u \rightarrow w$ in the orientation $O$, then $\sigma(u) \geq \sigma(w)$.
\end{thm}
The second characterization is in terms of hyperplane arrangements. 
Recall that for a simple graph $G=(V,E)$, we associate its \emph{real graphical arrangement} $\mathcal{A}_{\R}(G)\subset\R^V$, which is the union of the hyperplanes $$H_e:=\{\boldz = (z_v)_{v \in V}\in\R^V:z_v=z_w\}$$ over all edges $e=\{v,w\}\in E$. 
Zaslavsky proved the following (together with a generalization to arbitrary hyperplane arrangements\footnote{Las Vergnas gave a further generalization involving oriented matroids \cite{LasVergnas1980}.}):
\begin{thm}[{Zaslavsky, \cite[Thm. A]{Zaslavsky1975}}]
\label{thm:Zaslavsky}
The number of regions of $\calA_{\R}(G)$ (i.e. the number of connected components of $\R^V\setminus\mathcal A_{\R}(G)$) is precisely $(-1)^{|V|} \chi_G(-1)$.
\end{thm}
Greene and Zaslavsky \cite[Lem. 7.1]{GreeneZaslavsky} later gave a natural bijection between regions of $\calA_{\R}(G)$ and acyclic orientations of $G$, giving an alternate proof of the above theorem of Stanley. 

In fact, one may characterize arbitrary evaluations of $\chi_G$ at nonpositive integer values in terms of hyperplane arrangements. Fix an integer $m\ge3$, and let $\mathcal{A}_{\R}(G;m)\subset\R^V$ be the hyperplane arrangement obtained from $\mathcal{A}_{\R}(G)$ by adding\footnote{In this definition, the set $\{0,\ldots,m-2\}$ could be replaced with an arbitrary $(m-1)$-element set of real numbers, and everything that in this section would still hold. We have made a concrete choice for convenience in Section \ref{sec:hyperplanes}.} the hyperplanes $\{\boldz\in\R^V:z_v=i\}$ for every $v\in V$ and $i\in\{0,\ldots,m-2\}$. 
In Section \ref{sec:hyperplanes} we prove:

\begin{thm}
    $\calA_{\R}(G;m)$ has exactly $(-1)^{|V|} \chi_G( - (m - 2))$ bounded regions.
    \label{thm:num_bounded_components}
\end{thm}
Note that the case $m=3$ gives a characterization of $(-1)^{|V|}\chi_G(-1)$ slightly different from Zaslavsky's theorem above, but it is easily seen that they are equivalent; each region of $\calA_{\R}(G)$ intersects the unit hypercube $[0,1]^V$ in a unique (bounded) region of $\calA_{\R}(G;3)$, and all other regions of $\calA_{\R}(G;3)$ are unbounded.

The number of bounded regions of a hyperplane arrangement also has an interpretation in terms of critical points of a functional, due to Varchenko \cite{Varchenko1995}. For this,
let $\calA_{\R}\subset\R^n$ be a real hyperplane arrangement, and let $\calA_{\C}\subseteq\C^n$ denote the corresponding complex hyperplane arrangement. (We will sometimes write simply $\calA$ if the distinction between $\calA_{\R}$ and $\calA_{\C}$ is irrelevant.)
For every hyperplane $H \in \calA_{\R}$, let $f_H$ be a real degree-1 polynomial with zero set $H$. Let $\boldu = (u_H)_{H \in \calA} \in \C^{\# \calA}$, and consider the function
\begin{equation}\label{eq:LikelihoodFunction}
    F_{\calA, \boldu} = \prod_{H \in \calA} f_H^{u_H}.
\end{equation}
This is a \emph{multivalued} function on $\C^n \setminus \calA_{\C}$. However, the logarithmic derivatives $\frac{\partial \log F_{\calA, \boldu}}{\partial z_v}$ are independent of the branch, so it makes sense to talk about critical points of $F_{\calA, \boldu}$.
\begin{thm}[{Varchenko, \cite[Thm. 1.2.1]{Varchenko1995}, see also \cite{OrlikTerao,Huh2013}}]
    For general $\boldu\in\C^{\# \calA}$, the following hold:
    \begin{enumerate}
        \item $F_{\calA, \boldu}$ has finitely many critical points on $\C^n \setminus \calA_{\C}$, all of which are nondegenerate\footnote{A critical point is nondegenerate if the Hessian matrix at that point is nonsingular.}. \label{it:FinitelyManyCriticalPoints}
        \item The number of critical points is equal to the signed Euler characteristic $(-1)^{n}\chi(\C^n \setminus \calA_{\C})$.\label{it:CountCriticalPoints}
        \item If $u_H\in\R_{>0}$ for all hyperplanes $H$ in $\calA_{\R}$, then the critical points of $F_{\calA, \boldu}$ are all real, are all contained in bounded regions of $\R^n\setminus\calA_{\R}$, and there is exactly one critical point in each bounded region.\label{it:1PerRegion}
    \end{enumerate}
    \label{thm:varschenko}
\end{thm}

Theorem \ref{thm:varschenko}\eqref{it:1PerRegion} applied to $\calA(G;m)$, combined with Theorem~\ref{thm:num_bounded_components}, gives yet another interpretation of $(-1)^{|V|} \chi_G( - (m - 2))$. In fact, this interpretation is not too far from Theorem~\ref{thm:main}:

\begin{thm}
    For general $\boldu \in \C^{\# \calA(G;m)}$, $F_{\calA(G;m), \boldu}$ has exactly $\omega_{G,0,m}$ critical points.
    \label{thm:critical_point_interpretation}
\end{thm}
We prove Theorem \ref{thm:critical_point_interpretation} in Section~\ref{sec:hyperplanes}, as follows. By definition, the critical locus of $F_{\calA(G;m), \boldu}$ is an intersection of hypersurfaces: $$\Crit(F_{\calA(G;m), \boldu})=\bigcap_{v\in V}Z\left(\frac{\partial \log F_{\calA, \boldu}}{\partial z_v}\right)\subseteq\C^{V}$$ defined by the logarithmic derivatives of $F_{\calA, \boldu},$ where $Z(g)$ denotes the vanishing locus of $g$. We prove that under an appropriate embedding $\C^{V}\into\Mbar_{0,V\sqcup M}$, the intersection $\Crit(F_{\calA(G;m), \boldu})\subseteq\C^{V}$ is identified with the intersection of certain subvarieties of $\Mbar_{0,V\sqcup M}$; these subvarieties are representatives of the cohomology classes $\pi_M^*([pt])$ and $\pi_{N[v]\cup M}^*(\psi_v)$ (for all $v$) appearing in the definition of $\omega_{G,0,m}$. As is typical in intersection theory, some careful case analysis is required to match up the intersection number $\omega_{G,0,m}$ with the cardinality of an actual intersection.

Combining Theorem \ref{thm:critical_point_interpretation} with Theorem \ref{thm:num_bounded_components} and Varchenko's theorem above gives an alternative proof of Theorem \ref{thm:main} in the genus zero case (and in the higher-genus case, see Remark \ref{rem:HigherGenusFromGenusZero}):
\begin{proof}[Second proof of Theorem \ref{thm:main}, genus zero case]
    Let $U\subseteq\C^{\# \calA(G;m)}$ be a dense open set such that for $\boldu\in U,$ the conclusion of Theorem \ref{thm:varschenko} holds. Since the positive orthant $\R_{>0}^{\# \calA(G;m)}\subset\C^{\# \calA(G;m)}$ is Zariski-dense, we may fix $\boldu\in U\cap\R_{>0}^{\# \calA(G;m)}.$ Then
    \begin{align*}
        \omega_{G,0,m}&=\#\{\text{critical points of $F_{\calA(G;m), \boldu}$}\}&&\text{by Theorem \ref{thm:critical_point_interpretation}}\\
        &=\#\{\text{bounded regions of $\R^V\setminus\calA_{\R}(G;m)$}\}&&\text{by Theorem \ref{thm:varschenko}\eqref{it:1PerRegion}}\\
        &=(-1)^{\abs{V}}\chi_G(-(m-2))&&\text{by Theorem \ref{thm:num_bounded_components}}.\qedhere
    \end{align*}
\end{proof}

\subsection{Connections to algebraic statistics and to scattering amplitudes} \subsubsection{Algebraic Statistics} Theorem \ref{thm:critical_point_interpretation} interprets $\omega_{G,0,m}$ as counting critical points of a product of powers of linear functions. The problem of finding critical points of functions of this  form arises in statistics, specifically in maximum likelihood estimation (see \cite{HuhSturmfels}).
Maximum likelihood estimation is used to find a probability distribution within a fixed model (i.e. family of distributions) that best explains observed data. In this context, the hyperplanes $H\in\mathcal A$ correspond to possible outcomes, and the numbers $u_H$ are the observed frequencies of each outcome. The family of distributions, dependent on a vector of parameters $\boldz=(z_v)_{v\in V}$, is given by the function $(f_H)_{H\in\calA}:\C^{V}\setminus\calA\to\C^{\#\calA}.$ The function $F_{\calA,\boldu}$ measures the likelihood that a particular distribution would give rise to the observed data $\boldu=(u_H)_{H\in\calA},$ and is called the \emph{likelihood function}. One finds the critical points of this function in order to maximize it. The \emph{maximum likelihood degree} of the model is the number of critical points of the likelihood function, for generic data. 

Theorem \ref{thm:critical_point_interpretation} thus realizes $\omega_{G,0,m}$ as a maximum likelihood degree. The situation we consider, where the $f_H$s are all linear in the $z_v$s, correspond to models in which there are linear restrictions on the relative probabilities of various outcomes.

There are some differences between the usual setup of maximum likelihood estimation and our setup --- for example, in order to specify a model, one requires $\sum_{H\in\calA}f_H=1,$ while we do not require this.

\subsubsection{Scattering Amplitudes} In the case $m=3,$ Theorem \ref{thm:critical_point_interpretation} is also closely related to the Cachazo-He-Yuan model \cite{CHY} in the computation of scattering amplitudes. We identify $\M_{0,n+3}$ with the set of matrices of the form
\begin{equation*}
    X = \begin{bmatrix}
    1 & 1 & 1 & 1 & \dots & 1 & 0 \\
    0 & 1 & x_1 & x_2 & \dots & x_n & 1
    \end{bmatrix},
\end{equation*}
where all minors $p_{ij}, 1 \leq i < j \leq n+3$ are nonzero. In the CHY model, one considers the critical points of the \emph{scattering potential}
\begin{equation}
    L = \sum_{1 \leq i < j \leq n+3} s_{ij} \log (p_{ij})
\end{equation}
where $s_{ij}$ are called \emph{Mandelstam invariants}, satisfying $s_{ji} = s_{ij}, s_{ii} = 0$ and $\sum_{1\leq j \leq n+3} s_{ij} = 0$ (``conservation of momentum''). The space of all possible Mandelstam invariants is called \emph{kinematic space}. Scattering \emph{amplitudes}, in the CHY model, are sums of certain rational functions over all critical points of the scattering potential.

Note that the (nonconstant) minors $p_{ij}$ define hyperplanes in $\calA(K_n;3)$: indeed, $p_{1(j+2)} = x_j, p_{2(j+2)} = x_j - 1, p_{(i+2)(j+2)} = x_j - x_i$ for $1 \leq i,j \leq n$. We can therefore identify the scattering potential $L$ with the logarithm of $F_{\calA(K_n;3), \boldu}$. In particular, applying Theorem \ref{thm:critical_point_interpretation} to $G=K_n$ recovers the classical result (see \cite[Sec. 2]{SturmfelsTelen}) that, for generic Mandelstam invariants, the scattering potential has exactly $n!$ critical points in $\M_{0,n+3}$.

More generally, given a graph with index set $[n]$, we can consider the subspace of Mandelstam invariants where $s_{i+2,j+2} = 0 $ for nonadjacent $i,j\in[n]$. The number of critical points for generic Mandelstam invariants in this subspace is then given as $(-1)^n\chi_G(-1)$. Note that this never imposes restrictions on the Mandelstam invariants $s_{1,i}, s_{2,i}, s_{i,n+3}$, corresponding to the fact that our forgetful maps $\pi_{N[v]\sqcup M}$ never forget elements of $M$. Similar restrictions on Mandelstam invariants appear in \cite{CachazoEarly,EarlyPfisterSturmfels}.

\subsection{Further exploration}\label{sec:FurtherQuestions}
Theorem \ref{thm:main} raises many interesting questions. We outline some of these here, along with a few related observations, but there are doubtless many others that we are unaware of. Proofs of claims in this section are provided in Section \ref{sec:Digraphs}. 

\subsubsection{Directed graphs}
Attempting to generalize Theorem \ref{thm:main} to directed graphs (digraphs) yields some new and mysterious combinatorics. Note that the definition of $\Psi_{G,g,m}$ (Equation \eqref{eq:PsiDef} on page \pageref{eq:PsiDef}) works just as well for $G$ a (finite simple\footnote{This is, we allow two-cycles $u \rightleftarrows w$, but not parallel arrows $u \rightrightarrows w$ or loops.}) digraph, provided we replace the set of neighbors $N[v]$ with either the set of in-neighbors $N^{\mathrm{in}}[v]$ (vertices $w$ such that $w\to v$ is an edge of $G$) or out-neighbors $N^{\mathrm{out}}[v]$. For a digraph $G$, we therefore get \emph{two} classes $$\Psi_{G,g,m}^{\mathrm{in}},\Psi_{G,g,m}^{\mathrm{out}}\in H^*(\Mbar_{g,V\sqcup M})$$ and \emph{two} sets of intersection numbers $$\omega_{G,g,m}^{\mathrm{in}},\omega_{G,g,m}^{\mathrm{out}}$$ in exact analogy with \eqref{eq:PsiDef} and \eqref{eq:OmegaDef}. Reversing all edges of $G$ interchanges these pairs. The numbers $\omega_{G,g,m}^{\mathrm{in}}$ and $\omega_{G,g,m}^{\mathrm{out}}$ vary polynomially in $m$ as in the undirected case:
\begin{defthm}\label{defthm:DigraphPolynomials}
    There exist monic polynomials $\chi_G^{\mathrm{in}},\chi_G^{\mathrm{out}}\in\Z[x]$ of degree $\abs{V}$ such that, for $g\ge0$ and $m\ge0$ with $2g-2+m>0,$ we have
    \begin{align*}
        \omega_{G,g,m}^{\mathrm{in}}&=(-1)^{\abs{V}}\chi_G^{\mathrm{in}}(-(2g-2+m))\\
        \omega_{G,g,m}^{\mathrm{out}}&=(-1)^{\abs{V}}\chi_G^{\mathrm{out}}(-(2g-2+m)).\nonumber
    \end{align*}
    Furthermore, for the case $(g,m)=(1,0),$ we have
    \begin{align*}
        \omega_{G,1,0}^{\mathrm{in}}&=(-1)^{\abs{V}-1}\left.\deriv{}{x}\chi_G^{\mathrm{in}}(x)\right|_{x=0}&&\text{and}&
        \omega_{G,1,0}^{\mathrm{out}}&=(-1)^{\abs{V}-1}\left.\deriv{}{x}\chi_G^{\mathrm{out}}(x)\right|_{x=0}.
    \end{align*}
\end{defthm}
Definition/Theorem \ref{defthm:DigraphPolynomials}, which is proved in Section \ref{sec:Digraphs}, 
immediately raises many natural questions.
\begin{question}\label{q:DigraphQuestions}
    \begin{enumerate}
        \item Are $\chi_G^{\mathrm{in}}(x)$ and $\chi_G^{\mathrm{out}}(x)$ nonnegative for integers $x\ge0$?
        \item If so, do $\chi_G^{\mathrm{in}}(x)$ and $\chi_G^{\mathrm{out}}(x)$ have combinatorial interpretations?
        \item Do $\chi_G^{\mathrm{in}}$ and $\chi_G^{\mathrm{out}}$ have alternating coefficients?
        \item Do $\chi_G^{\mathrm{in}}$ and $\chi_G^{\mathrm{out}}$ satisfy recursions along the lines of deletion-contraction?
        \item Do $\chi_G^{\mathrm{in}}$ and $\chi_G^{\mathrm{out}}$ satisfy other properties of chromatic polynomials, e.g. most notably log-concavity of coefficients as per \cite{Huh2012}?
    \end{enumerate}
\end{question}

As an initial observation, we have the following recursion, which gives a closed formula for $\chi_G^{\mathrm{in}}$ and $\chi_G^{\mathrm{out}}$ when $G$ has no directed cycles.
\begin{prop}\label{prop:SourceRecursion}
    Let $G=(V,E)$ be a finite simple directed graph, let $v\in V$, and let $G\setminus\{v\}$ be the graph obtained by deleting $v$. If $v$ is a sink, then
    \begin{align*}
        \chi_G^{\mathrm{in}}(x)=(x-\indeg(v))\cdot\chi_{G\setminus v}^{\mathrm{in}}(x).
    \end{align*}
    If $v$ is a source, then
    \begin{align*}
        \chi_G^{\mathrm{out}}(x)=(x-\outdeg(v))\cdot\chi_{G\setminus v}^{\mathrm{out}}(x).
    \end{align*}
    Here $\indeg(v)$ and $\outdeg(v)$ denote the number of incoming and outgoing edges at $v$, respectively.
\end{prop}
An acyclic directed graph always has at least one sink (vertex with no outgoing edges) and at least one source (vertex with no incoming edges). Proposition \ref{prop:SourceRecursion} therefore immediately implies the following by induction:
\begin{cor}\label{cor:AcyclicDigraph}
    Let $G=(V,E)$ be an acyclic finite simple directed graph. Then
    \begin{align*}
        \chi_G^{\mathrm{in}}(x)&=\prod_{v\in V}(x-\indeg(v))&&\text{and}&
        \chi_G^{\mathrm{out}}(x)&=\prod_{v\in V}(x-\outdeg(v)).
    \end{align*}
    
\end{cor}
Proposition \ref{prop:SourceRecursion} is proved in Section \ref{sec:Digraphs}.
\begin{ex}\label{ex:Digraph}
    For the digraph shown in Figure \ref{fig:Digraph}, we have 
    \begin{align*}
        \chi_G^{\mathrm{in}}(x)& = x^2(x-2) =x^3-2x^2&&\text{and}&
        \chi_G^{\mathrm{out}}(x)& = x(x-1)^2 =x^3-2x^2+x.
    \end{align*}
    
    \begin{figure}
        \centering
        \begin{tikzpicture}
            \draw[->,very thick] (0.1,0)--(0.9,0);
            \draw[->,very thick] (1.9,0)--(1.1,0);
            \draw (0,0) node {$\bullet$};
            \draw (1,0) node {$\bullet$};
            \draw (2,0) node {$\bullet$};
        \end{tikzpicture}
        \caption{The smallest digraph for which $\chi_G^{\mathrm{in}}\ne\chi_G^{\mathrm{out}}$, see Example \ref{ex:Digraph}.}
        \label{fig:Digraph}
    \end{figure}
\end{ex}

We will establish many combinatorial properties of $\chi_G^{\mathrm{in}}$ and $\chi_G^{\mathrm{out}}$ in the forthcoming paper \cite{ReinkeSilversmith2026}.

\begin{remark}
    There is a notion of \emph{chromatic polynomial of a digraph} \cite{Harutyunyan2011,AkbariGhodratiJabalameliSaghafian2017,HochstattlerWiehe2021,GonzalezHernandezOrtizLlanoOlsen2022}. These polynomials count digraph colorings in the sense of Neumann-Lara \cite{NeumannLara1982}, where each color must induce an acyclic subdigraph. There is no obvious relationship to $\chi_G^{\mathrm{in}}$ and $\chi_G^{\mathrm{out}}$; for example, the digraph in Figure \ref{fig:Digraph} has no directed cycles, hence has chromatic polynomial $x^3$ in this sense.

    (There is another, inequivalent, notion of digraph coloring that applies only to digraphs with no pairs of opposite edges, which also yields a ``oriented chromatic polynomial'', see \cite{Sopena2016}. Again, there appears to be no relation to $\chi_G^{\mathrm{in}}$ and $\chi_G^{\mathrm{out}}$.)
\end{remark}

\subsubsection{Other intersection numbers involving \texorpdfstring{$\Psi_{G,g,m}$}{Psi\_{G,g,m}}} It would be interesting to know how much combinatorial information about $G$ is captured in $\Psi_{G,g,m}.$ By Theorem \ref{thm:main}, the classes $\Psi_{G,g,m}$ collectively ``know about'' the chromatic polynomial of $G$, in the sense that $\chi_G$ is determined by its values at negative integers --- but it may know more. 
\begin{question}
    If $\Psi_{G_1,g,m}=\Psi_{G_2,g,m}$ for all $g$ and $m$, what can one conclude about the relationship between $G_1$ and $G_2$? Must they be isomorphic? Must they have the same graphic matroid? Must they have the same Tutte polynomial? 
\end{question}
Perhaps relatedly, besides \eqref{eq:OmegaDef}, there are other natural ways to define an intersection number from $\Psi_{G,g,m}$. For example, recall the tautological ``kappa-classes'' $$\kappa_a=(\pi_{12\cdots n})_*(\psi_{n+1}^{a+1})\in H^{2a}(\Mbar_{g,n}),$$ pushed forward along the map $\pi_{12\cdots n}:\Mbar_{g,n+1}\to\Mbar_{g,n}.$ We define
\begin{align*}
    \omega^{\kappa}_{G,g,m}=24^g\cdot g!\cdot\int_{\Mbar_{g,V\sqcup M}}\Psi_{G,g,m}\cdot\kappa_{3g-3+m},
\end{align*}
where the factor is, like in \eqref{eq:Omega10}, a rather natural normalization.
\begin{question}
    Find a combinatorial interpretation (or a closed formula) for $\omega^{\kappa}_{G,g,m}$.
\end{question}
One may prove the following formula for graphs with no edges --- it is striking that even the edgeless case is somewhat nontrivial:
\begin{prop}
    Let $G_n$ be a graph with $n$ vertices and no edges. Then $$\omega^\kappa_{G,g,m}=\sum_{i=0}^{3g-2+m}(2g-2+m)^{n-i}\cdot\binom{n}{i}.$$
\end{prop}
\begin{proof}[Proof sketch]
    Pull back to $\Mbar_{g,n+m+1}$ to replace $\kappa_{3g-3+m}$ with a $\psi$-class, then use Facts \ref{fact:PsiPullback}--\ref{fact:Surplus} as in the proof of Theorem \ref{thm:main} to prove the recursion \begin{align*}
        \omega_{G_n,g,m}&=(2g-2+m+1)\omega_{G_{n-1},g,m}-\binom{n-1}{3g-3+m+1}(2g-2+m)^{n-1-(3g-3+m+1)},
    \end{align*}
    from which the Proposition follows by a straightforward induction.
\end{proof}
One may check that $\omega^{\kappa}_{G,g,m}$ does \emph{not} satisfy the deletion-contraction formula.
\begin{question}
    Does $\omega^{\kappa}_{G,g,m}$ satisfy a recursion similar to the deletion-contraction formula? Does $\Psi_{G,g,m}$?
\end{question}
\begin{remark}
    One might interpolate between $\omega_{G,g,m}$ and $\omega^{\kappa}_{G,g,m}$ by considering numbers of the form $$\int_{\Mbar_{g,V\sqcup M}}\Psi_{G,g,m}\cdot\kappa_a\cdot\pi_{M}^*(\kappa_{3g-3+m-a}).$$
\end{remark}

\subsubsection{Generalization to matroids}
A matroid $M$ has a characteristic polynomial $p_M(k)$, which specializes (after a small correction) to the chromatic polynomial of a graph in the case of a graphic matroid, and which satisfies a deletion-contraction formula.
\begin{question}
    Do there exist classes $\Psi_{M,g,m}\in A^*(\Mbar_{g,n})$ (for appropriate $n$), and corresponding intersection numbers $\omega_{M,g,m}$, that recover the values of $p_M(k)$ at negative integers?
\end{question}

\subsection{Acknowledgements} The first author would like to thank Leonie Kayser and  Andreas Kretschmer for helpful discussions related to the likelihood correspondence. The second author is grateful to R. Cavalieri, A. Pixton, R. Ramadas, and D. Ross for useful discussions and suggestions --- in particular around the base case of Theorem \ref{thm:main} (Ramadas), the genus-1 case (Cavalieri), and the proof of Definition/Theorem \ref{defthm:DigraphPolynomials} (Pixton, Ross). We also thank the anonymous referees for helpful suggestions on how to improve the exposition. This project began during the second author's visit to the Max Planck Institute for Mathematics in the Natural Sciences, in Leipzig, for the conference ``Combinatorial algebraic geometry from physics''. The second author thanks the MPI for their hospitality.

\section{Basic geometry and intersection theory of \texorpdfstring{$\Mbar_{g,n}$}{M\_{g,n}-bar}}\label{sec:MgnBasics}
Here we summarize the facts that are needed for the proof of Theorem \ref{thm:main}. Everything in this section is well-known --- see \cite{KockNotes} for details unless otherwise specified.

Let $g$ be a nonnegative integer, and let $P$ be a finite set such that $2g-2+\abs{P}>0.$ Let $\Mbar_{g,P}=\Mbar_{g,\abs{P}}$ denote the moduli space of stable curves of genus $g$ with marked points indexed by $P$.

\begin{notation}\label{not:BasicMgnNotation} We will use the following notations:
\begin{enumerate}
    \item For $P'\subseteq P$ with $2g-2+\abs{P'}>0,$ let $$\pi_{P'}:\Mbar_{g,P}\to\Mbar_{g,P'}$$ denote the (flat) ``\textbf{forgetful morphism}'' that removes the marked points \emph{not} in $P'$, and stabilizes the curve if necessary.
    \item For $i\in P,$ let $\psi_i$ denote the \textbf{cotangent class} at the $i$th marked point.
    \item A $P$-marked stable genus-$g$ curve $C$ has a dual graph $\Gamma$, whose vertices $\alpha$ correspond to the irreducible components $C_\alpha\subseteq C$, and whose edges correspond to nodes of $C$. There is a natural ``weight function'' on the vertices that assigns to $\alpha$ the  genus $g_\alpha=g(C_\alpha)$, and a natural way to attach extra ``half-edges'' indexed by $P$ to vertices of $\Gamma$. The resulting structure is a ``\textbf{stable $P$-marked genus-$g$ weighted graph}''. That is, $\Gamma$ satisfies $b_1(\Gamma)+\sum_{\alpha\in V(\Gamma)}g_\alpha=g,$ where $b_1(\Gamma)$ is the first Betti number of $\Gamma$, and every vertex $\alpha\in V(\Gamma)$ satisfies $2g_\alpha-2+\val(\alpha)>0,$ where $\val(\alpha)$ includes the half-edges attached to $\alpha$. (See e.g. \cite[App. A]{GraberPandharipande2003}.) We write $\HE(\Gamma,\alpha)$ for the set of half-edges of $\Gamma$ incident to a vertex $\alpha\in V(\Gamma)$ (also called ``flags'' in the literature), which includes marked half-edges attached to $\alpha$ as well as germs of actual edges.
    \item For any stable $P$-marked genus-$g$ weighted graph $\Gamma,$ define $$\Mtilde_{\Gamma}=\prod_{\alpha\in V(\Gamma)}\Mbar_{g_\alpha,\HE(\Gamma,\alpha)}.$$ There is a natural morphism $\iota_\Gamma:\Mtilde_{\Gamma}\to\Mbar_{g,P}$ whose image is the (closed) \textbf{boundary stratum} $\Mbar_{\Gamma}\subseteq\Mbar_{g,P}$ corresponding to $\Gamma,$ i.e. the closure of the locus of stable curves with dual graph $\Gamma.$ The map $\iota_\Gamma$ induces an isomorphism $\Mtilde_\Gamma/\Aut(\Gamma)\to\Mbar_\Gamma$. 
    \label{not:DecomposeBoundaryStratum}
    \item In the case $$\Gamma=\begin{tikzpicture}[baseline=-0.65ex]
            \draw (0,0) -- (1,0);
            \foreach \th in {-60,-40,-20,0,20,40,60} {
            \draw (0,0)--++(180+\th:.6);
            };
            \draw (-.6,0) node[left] {$I$};
            \foreach \th in {-60,-40,-20,0,20,40,60} {
            \draw (1,0)--++(\th:.6);
            };
            \draw (1.6,0) node[right] {$P\setminus I$};
            \filldraw[fill=white] (0,0) circle(.25);
            \filldraw[fill=white] (1,0) circle(.25);
            \draw (0,0) node {$0$};
            \draw (1,0) node {$g$};
        \end{tikzpicture}$$ where $I\subseteq P$ with $\abs{I}\ge2$, we use the special notation $D_{I}:=\Mbar_\Gamma$ and $\iota_{I}:=\iota_\Gamma$.\label{not:BoundaryDivisor} We abuse notation slightly and also denote by $D_I$ the corresponding divisor \emph{class}.
    \item Below we will also introduce pictorial notations for certain sums of (pushforwards of $\psi$-classes from) boundary strata. The most basic case is that we write the stable graph $\Gamma$ as a placeholder for the class $[\Mbar_{\Gamma}]\in H^*(\Mbar_{g,P}).$
\end{enumerate}
\end{notation}

\begin{remark}\label{rem:Stacks}
    Some stack-theoretic carefulness is needed to work with boundary strata $\Mbar_{\Gamma}$ such that $\Aut(\Gamma)$ is nontrivial --- however, this is essentially irrelevant for us, as almost all boundary strata used in this paper manifestly have trivial $\Aut(\Gamma)$. (The only exception is in the base case of the proof of Theorem \ref{thm:main}.)
\end{remark}

\begin{fact}[Pulling back $\psi$-classes along forgetful maps, {\cite[Lem. 1.3.1]{KockNotes}}]\label{fact:PsiPullback}
    For distinct $i,j\in P,$ the following ``comparison theorem'' relates $\psi_i\in H^2(\Mbar_{g,P})$ to the pullback of $\psi_i\in H^2(\Mbar_{g,P\setminus\{j\}})$ along the map that forgets $j$: $$\psi_i=\pi_{P\setminus\{j\}}^*\psi_i+D_{\{ij\}}.$$ Pictorially, we write
    \begin{align}\label{eq:PsiPullbackSingle}
        \boxed{\psi_i\quad=\quad\pi_{P\setminus\{j\}}^*\psi_i\quad+\quad
        \begin{tikzpicture}[baseline=-0.65ex]
            \draw (0,0) -- (1,0);
            \draw (0,0)--++(135:.6);
            \draw (135:.6) node[left] {$i$};
            \draw (0,0)--++(225:.6);
            \draw (225:.6) node[left] {$j$};
            \foreach \th in {-60,-40,-20,0,20,40,60} {
            \draw (1,0)--++(\th:.6);
            };
            \draw (1.6,0) node[right] {$P\setminus\{i,j\}$};
            \filldraw[fill=white,thick] (0,0) circle(.25);
            \filldraw[fill=white,thick] (1,0) circle(.25);
            \draw (0,0) node {$0$};
            \draw (1,0) node {$g$};
        \end{tikzpicture}}
    \end{align}
    Applying \eqref{eq:PsiPullbackSingle} repeatedly, we get a comparison formula for pullbacks of $\psi$-classes along general forgetful maps. Let $P'\subseteq P$ with $i\in P'$, such that $2g-2+\abs{P'}>0.$ Then $$\psi_i=\pi_{P'}^*\psi_i+\sum_{I}D_{I},$$ where $I$ ranges over subsets of $P\setminus P'$ such that $i\in I$ and $\abs{I}\ge2$. We depict this identity pictorially as
    \begin{align*}
        \boxed{\psi_i\quad=\quad\pi_{P'}^*\psi_i\quad+\quad
        \begin{tikzpicture}[baseline=-0.65ex]
            \draw (0,0) -- (1,0);
            \draw (0,0)--++(135:.6);
            \draw (135:.6) node[left] {$i$};
            \foreach \th in {-60,-40,-20,0,20,40,60} {
            \draw (1,0)--++(\th:.6);
            };
            \draw (1.6,0) node[right] {$P'\setminus i$};
            \filldraw[fill=white,thick] (0,0) circle(.25);
            \filldraw[fill=white,thick] (1,0) circle(.25);
            \draw (0,0) node {$0$};
            \draw (1,0) node {$g$};
        \end{tikzpicture}}
    \end{align*}
    Here the picture means ``sum over the ways to attach all unassigned marked points at vertices of the graph, such that the resulting graph is stable.''
\end{fact}
\begin{fact}[Pulling back boundary strata, see {\cite[Fact 3]{Keel1992}}]\label{fact:BoundaryPullback}
    Let $P'\subseteq P$, and let $I\subseteq P'$ with $2\le\abs{I}\le\abs{P'}-2$. Then $$\pi_{P'}^*D_I=\sum_{J}D_J,$$ where $J$ ranges over subsets of $P$ containing $I$ and \emph{no other} elements of $P'$. Using the same pictorial system as above (with the same implicit sum taken), we get:
    \begin{align*}
        \boxed{\pi_{P'}^*D_I\quad=\quad
        \begin{tikzpicture}[baseline=-0.65ex]
            \draw (0,0) -- (1,0);
            \draw (0,0) node {$\bullet$};
            \draw (1,0) node {$\bullet$};
            \foreach \th in {-60,-40,-20,0,20,40,60} {
            \draw (1,0)--++(\th:.6);
            \draw (0,0)--++(180+\th:.6);
            };
            \draw (1.6,0) node[right] {$P'\setminus I$};
            \draw (-.6,0) node[left] {$I$};
            \filldraw[fill=white,thick] (0,0) circle(.25);
            \filldraw[fill=white,thick] (1,0) circle(.25);
            \draw (0,0) node {$0$};
            \draw (1,0) node {$g$};
        \end{tikzpicture}}
    \end{align*}
    This formula generalizes to pullbacks of arbitrary boundary strata. Let $\Gamma$ be a stable $P'$-marked genus-$g$ weighted graph. Then 
    \begin{align}\label{eq:PullbackOfBoundaryStrata1}
        \pi_{P'}^*(\iota_{\Gamma})_*([\Mtilde_\Gamma])=\sum_{\Lambda}(\iota_{\Lambda})_*([\Mtilde_{\Lambda}]),
    \end{align}
    where $\Lambda$ runs over the set of stable $P$-marked genus-$g$ graphs obtained by attaching the elements of $P\setminus P'$ to vertices of $\Gamma.$ This may be stated in terms of classes of boundary strata: 
     \begin{align*}
         \pi_{P'}^*([\Mbar_{\Gamma}])=\sum_{\Lambda}\frac{\lvert\Aut(\Lambda)\rvert}{\abs{\Aut\Gamma}}[\Mbar_{\Lambda}].
     \end{align*} 
     Fact \ref{fact:BoundaryPullback} is well-known to experts, but we do not know of a proof in the literature. The genus-zero case is proved in \cite[Fact 3]{Keel1992}, and the general proof is similar.
    \begin{proof}[Proof of Fact \ref{fact:BoundaryPullback}]
        By induction we may assume that $\abs{P'}=\abs{P}-1.$ Then a stable graph $\Lambda$ as above is specified by the vertex of $\Gamma$ where we attach the new marking. By \cite[Sec. 2]{Knudsen1983}, the map $\pi_{P'}:\Mbar_{g,P}\to\Mbar_{g,P'}$ is naturally identified with the universal curve $\mathcal C_{g,P'}\to\Mbar_{g,P'}$ over $\Mbar_{g,P'}$, and similarly $\bigsqcup_{\Lambda}\Mtilde_\Lambda\to\Mtilde_\Gamma$ is naturally identified with $\bigsqcup_{\alpha\in V(\Gamma)}\pr_\alpha^*\mathcal C_{g_\alpha,\HE(\Gamma,\alpha)}$, where $\pr_\alpha:\Mtilde_\Gamma\to\Mbar_{g_\alpha,\HE(\Gamma,\alpha)}$ is the projection onto the $\alpha$-th factor. 

        Thus we have a commutative diagram
        $$\begin{tikzcd}
            {\displaystyle \bigsqcup_{\alpha\in V(\Gamma)}\pr_\alpha^*\mathcal C_{g_\alpha,\HE(\Gamma,\alpha)}}\arrow[r,"\phi_1"]
            \arrow[d,equal]&\iota_{\Gamma}^*\mathcal C_{g,P'}\arrow[r]\arrow[d,equal]&\mathcal C_{g,P'}\arrow[d,equal]\\
            {\displaystyle \bigsqcup_{\Lambda}\Mtilde_\Lambda}\arrow[r,"\phi_2"]\arrow[dr]&\iota_{\Gamma}^*\Mbar_{g,P}\arrow[r,"\phi_3"]\arrow[d]&\Mbar_{g,P}\arrow[d,"\pi_{P'}"]\\
            &\Mtilde_\Gamma\arrow[r,"\iota_\Gamma",swap]&\Mbar_{g,P'}
        \end{tikzcd}$$
        where the lower right square is Cartesian. It is clear from these identifications (1) that $\phi_3\circ\phi_2=\bigsqcup_{\Lambda}\iota_\Lambda$, and (2) that the map labeled $\phi_1$ is simply the normalization map that glues together the irreducible components of $\iota_\Gamma^*\mathcal C_{g,P'}$. In particular, $\phi_2$ is birational, so the push-pull formula gives:
    \begin{align*}
        \pi_{P'}^*(\iota_{\Gamma})_*([\Mtilde_\Gamma])&=(\phi_3)_*([\iota_\Gamma^*\mathcal C_{g,P'}])\\
        &=\sum_{\Lambda}(\phi_3\circ\phi_2)_*([\Mtilde_\Lambda])\\
        &=\sum_{\Lambda}(\iota_{\Lambda})_*([\Mbar_{\Lambda}]).\qedhere
    \end{align*}

\end{proof}
    
\end{fact}
\begin{fact}[Restricting $\psi$-classes to boundary strata]\label{fact:PsiRestrict}
    The restriction of $\psi_i$ to a boundary stratum $\Mbar_{\Gamma}$ is given by ``$\psi_i$ on the vertex of $\Gamma$ containing $i$.'' Precisely, let $\Gamma$ be a stable $P$-marked genus-$g$ weighted graph, let $i\in P$, and let $\alpha_i\in V(\Gamma)$ be the vertex at which $i$ is attached. Then
    \begin{align}\label{eq:PsiRestrict1}
        \iota_\Gamma^*\psi_i=\pr_{\alpha_i}^*\psi_i\in H^2(\Mtilde_\Gamma),
    \end{align} where $\pr_{\alpha_i}:\Mtilde_\Gamma\to\Mbar_{g_{\alpha_i},\HE(\Gamma,\alpha_i)}$ is the projection. Note that on the left side of \eqref{eq:PsiRestrict1},  $\psi_i$ is a class in $H^2(\Mbar_{g,P})$, whereas on the right side, $\psi_i$ is a class in $H^2(\Mbar_{g_{\alpha_i},\HE(\Gamma,\alpha_i)})$.
    
    More generally, the restriction of $\pi_{P'}^*\psi_i$ to $\Mbar_{\Gamma}$ is the $\psi$-class of the half-edge where $i$ lands under \emph{forgetting and stabilizing}, pulled back along the appropriate forgetful map. This is easiest to understand pictorially, and is illustrated in Figure \ref{fig:PsiRestrict}; we now introduce the notation required to state it precisely. 
    
    Let $P'\subseteq P$ be a subset with $2g-2+\abs{P'}>0,$ and let $i\in P'.$ Let $\Gamma'$ denote the stable $P'$-marked genus-$g$ graph obtained from $\Gamma$ by forgetting the marks in $P\setminus P'$ and stabilizing (i.e. successively contracting edges incident to vertices with weight 0 of valence $\le2$, see e.g. \cite[Sec. 1.3.3]{KockVainsencher}). 
    Stabilization induces a natural weight-preserving inclusion of vertex sets $$\st_{\Gamma'}^{\mathrm{vert}}:V(\Gamma')\into V(\Gamma),$$ and for any vertex $\epsilon\in V(\Gamma'),$ induces a natural inclusion $$\st_{\Gamma',\epsilon}^{\mathrm{he}}:\HE(\Gamma',\epsilon)\to\HE(\Gamma,\st_{\Gamma'}^{\mathrm{vert}}(\epsilon)).$$ Let $\alpha_i'\in V(\Gamma')$ be the vertex at which the marked half-edge $i$ is attached, let $$\alpha_{P',i}=\st_{\Gamma'}^{\mathrm{vert}}(\alpha_i')\in V(\Gamma),$$ and let $$q_{P',i}=\st_{\Gamma',\alpha_i'}^{\mathrm{he}}(i)\in\HE(\Gamma,\alpha_{P',i}).$$ Let $Q_{P',i}\subseteq\HE(\Gamma,\alpha_{P',i})$ denote the set of half-edges that come from $\Gamma'$, i.e. $Q_{P',i}$ is the image of $\st_{\Gamma',\alpha_i'}^{\mathrm{he}}$. (In particular, $q_{P',i}\in Q_{P',i}$.) Then 
    \begin{align}\label{eq:PsiRestrict2}
        \iota_\Gamma^*(\pi_{P'}^*\psi_i)=\pr_{\alpha_{P',i}}^*\pi_{Q_{P',i}}^*\psi_{q_{P',i}}\in H^2(\Mtilde_\Gamma).
    \end{align}
    
    \begin{figure}
        \centering
        \begin{tikzpicture}[scale=1.5]
            \draw[very thick] (0,0) -- (2,0);
            \draw[very thick] (1,0)--(1,-1);
            \draw[very thick] (1,-1)--++(-30:1);
            \draw (0,0) node {$\bullet$};
            \draw (1,0) node {$\bullet$};
            \draw (2,0) node {$\bullet$};
            \draw (1,-1) node {$\bullet$};
            \draw (1,-1)++(-30:1) node {$\bullet$};
            \draw[very thick,densely dashed] (1,-1)++(-30:1)--++(0:.6);
            \draw (1,-1)++(-30:1)--++(-60:.6);
            \foreach \th in {-60,-30,0,30,60} {
            \draw[very thick,densely dashed] (0,0)--++(180+\th:.6);
            };
            \foreach \th in {-60,-20,20,60} {
            \draw[very thick,densely dashed] (2,0)--++(\th:.6);
            };
            \draw[very thick,densely dashed] (1,-1)--++(-90-60:.6);
            \draw[very thick,densely dashed] (1,-1)--++(-90:.6);
            \draw[very thick,densely dashed] (1,-1)--++(-90+20:.6);
            \draw[red,very thick] (1,-1)--++(-90+60:.5);
            \draw[red,very thick] (1,-1)--++(90:.5);
            \draw (240:.6) node[below] {$i$};
            \draw[line width=3 pt] (240:.6)--(0,0);
            \draw[line width=3 pt,red] (1,-1)--++(-90-40:.6);
            \draw[line width=3 pt,red] (1,-1)--++(-90-20:.6);
            \draw[line width=3 pt,red] (1,-1)--++(-90+40:.6);
            \draw[line width=3 pt] (1,-1)++(-90+60:1)--++(-60:.6);
            \draw (2.5,-1) node {$\alpha_{P',i}$};
            \draw[->] (2.2,-.95) to[out=170,in=10] (1.2,-.9);
            \draw (-.5,-1.5) node {$q_{P',i}$};
            \draw[->] (-.4,-1.4) to[out=45,in=180] (.95,-.7);
            \filldraw[fill=white,very thick] (0,0) circle(.17);
            \filldraw[fill=white,very thick] (1,0) circle(.17);
            \filldraw[fill=white,very thick] (2,0) circle(.17);
            \filldraw[fill=white,very thick] (1,-1) circle(.17);
            \filldraw[fill=white,very thick] (1,-1)++(-30:1) circle(.17);
            \draw (0,0) node {$0$};
            \draw (1,0) node {$0$};
            \draw (2,0) node {$0$};
            \draw (1,-1) node {$0$};
            \draw (1,-1)++(-30:1) node {$1$};
            \draw (1,-2.2) node {\Large$\Gamma$};
            \draw (4,-1.2) node {\Large$\hookleftarrow$};
        \end{tikzpicture}
        \quad\quad\quad\quad
        \begin{tikzpicture}[scale=1.5]
            \draw[very thick] (1,-1)--++(-30:1);
            \draw (1,-1) node {$\bullet$};
            \draw (1,-1)++(-30:1) node {$\bullet$};
            \draw (2.5,-1) node {$\alpha_i'$};
            \draw[->] (2.35,-.95) to[out=170,in=10] (1.2,-.9);
            \draw (1,-1)++(90:.6) node[above] {$i$};
            \draw[line width=3 pt] (1,-1)--++(-90-40:.6);
            \draw[line width=3 pt] (1,-1)--++(-90-20:.6);
            \draw[line width=3 pt] (1,-1)--++(-90+40:.6);
            \draw[line width=3 pt] (1,-1)--++(90:.6);
            \draw[line width=3 pt] (1,-1)++(-90+60:1)--++(-60:.6);
            \filldraw[fill=white,very thick] (1,-1) circle(.17);
            \filldraw[fill=white,very thick] (1,-1)++(-30:1) circle(.17);
            \draw (1,-1) node {$0$};
            \draw (1,-1)++(-30:1) node {$1$};
            \draw (1,-2.2) node {\Large$\Gamma'$};
        \end{tikzpicture}
        \caption{An illustration of Fact \ref{fact:PsiRestrict}. We are pulling back $\pi_{P'}^*\psi_i$ along the boundary stratum map $\iota_\Gamma:\Mtilde_\Gamma\to\Mbar_{g,P}$, where $\Gamma$ is the stable marked genus-1 weighted graph on the left, $P'\subset P$ is the set of solid marked half-edges. The stabilization $\Gamma'$ is shown on the right. The class $\iota_\Gamma^*\pi_{P'}^*\psi_i$ lives on the vertex $\alpha_{P',i}$, where it is the $\psi$-class at the half-edge $q_{P',i}$, pulled back along the forgetful map that remembers the red half-edges (since those are the ones that survive under stabilization). In particular, $$\Mtilde_\Gamma\cong\Mbar_{0,6}\times\Mbar_{0,3}\times\Mbar_{0,5}\times\Mbar_{0,8}\times\Mbar_{1,3},$$ and $\iota_\Gamma^*\pi_{P'}^*\psi_i$ is pulled back from $\Mbar_{0,8}$. In our notation, the inclusions $\st_{P'}^\mathrm{vert}$ and $\st_{P',\epsilon}^{\mathrm{he}}$ (for both choices of $\epsilon\in V(\Gamma')$) are ``translation to the left''.}
        
        
        \label{fig:PsiRestrict}
    \end{figure}
    We will often use Fact \ref{fact:PsiRestrict} in the following form, where we have used the projection formula: $$(\pi_{P_1'}^*\psi_{i_1})\cdots(\pi_{P_r'}^*\psi_{i_r})\cdot(\iota_\Gamma)_*([\Mtilde_{\Gamma}])=(\iota_\Gamma)_*((\pr_{\alpha_{P_1',i_1}}^*\pi_{Q_{P_1',i_1}}^*\psi_{q_{P_1',i_1}})\cdots(\pr_{\alpha_{P_r',i_r}}^*\pi_{Q_{P_r',i_r}}^*\psi_{q_{P_r',i_r}})).$$

    We provide a (straightforward) proof of \eqref{eq:PsiRestrict2}, since we do not know a reference. (The case where $g=0$ and $\Gamma$ has one edge is \cite[Lem. 3.2]{BrakensiekEurLarsonLi2023}, and the proof is similar.)
    \begin{proof}[Proof of Fact \ref{fact:PsiRestrict}]
        By definition, $\psi_i$ is the first Chern class of the line bundle $\mathcal{L}_i$ whose fiber at a $P$-marked curve $C$ is the cotangent space to $C$ at the marked point $i$. It is clear from this description that \begin{align}\label{eq:PsiRestrict3}
            \iota_\Gamma^*\mathcal{L}_i=\pr_{\alpha_i}^*\mathcal{L}_i,
        \end{align} where $\alpha_i$ is the vertex of $\Gamma$ containing half-edge $i$. Note that in \eqref{eq:PsiRestrict3}, the $\mathcal L_i$ on the left is a line bundle over $\Mbar_{g,P},$ whereas the $\mathcal L_i$ on the right is a line bundle over $\Mbar_{g_{\alpha_i},\HE(\Gamma,\alpha_i)}$. This immediately implies \eqref{eq:PsiRestrict1}.

        The argument for \eqref{eq:PsiRestrict2} is similar. For brevity, we write 
        \begin{align*}
            \alpha&:=\alpha_{P',i}&
            q&:=q_{P',i}&&\text{and}&
            Q&:=Q_{P',i}.
        \end{align*}
        We have a commutative diagram
        $$\begin{tikzcd}
            &\Mbar_{g_{\alpha},\HE(\Gamma,\alpha)}\arrow[dl,"\pi_Q",swap]\arrow[d,"\pi_{P',\alpha}"]&\Mtilde_{\Gamma}\arrow[r,"\iota_\Gamma"]\arrow[d,"\widetilde{\pi_{P'}}"]\arrow[l,"\pr_{\alpha}",swap]&\Mbar_{g,P}\arrow[d,"\pi_{P'}"]\\
            \Mbar_{g_{\alpha_i'},Q}\arrow[r,equal]&\Mbar_{g_{\alpha_i'},\HE(\Gamma',\alpha_i')}&\Mtilde_{\Gamma'}\arrow[r,"\iota_{\Gamma'}"]\arrow[l,"\pr_{\alpha_i'}",swap]&\Mbar_{g,P'}
        \end{tikzcd}$$
        Here $\widetilde{\pi_{P'}}$ is a product of forgetful maps, together with projection along the factors $\alpha\in V(\Gamma)$ corresponding to vertices of $\Gamma$ not in $\Gamma'$, and $\pi_{P',\alpha}$ renames each half-edge in $Q$ to the corresponding half-edge in $\Gamma'$ based at $\alpha_i'$, and forgets the half-edges not in $Q$.
        Thus 
        \begin{align*}
            \iota_\Gamma^*\pi_{P'}^*\psi_i&=\widetilde{\pi_{P'}}^*\iota_{\Gamma'}^*\psi_i\\
            &=\widetilde{\pi_{P'}}^*\pr^*_{\alpha_i'}\psi_i\quad\quad\quad\quad\text{by \eqref{eq:PsiRestrict3}}\\
            &=\pr_{\alpha}^*\pi_{P',\alpha}^*\psi_i\\
            &=\pr_{\alpha}^*\pi_{Q}^*\psi_q.\qedhere
        \end{align*}
    \end{proof}
    
\end{fact}
\begin{fact}[Writing $\psi$-classes as sums of boundary divisors in genus zero, {\cite[Sec. 1.5.2]{KockNotes}}]\label{fact:WritePsiAsBoundary}
    Assume $g=0$, and let $i,j,k\in P$ be distinct. Then in $H^2(\Mbar_{0,P})$ we have the identity (with an implicit summation as above) $$\boxed{\psi_i\quad=\quad\begin{tikzpicture}[baseline=-0.65ex]
            \draw (0,0) -- (1,0);
            \draw (0,0) node {$\bullet$};
            \draw (1,0) node {$\bullet$};
            \draw (0,0)--++(135:.6);
            \draw (135:.6) node[left] {$i$};
            \foreach \th in {-45,45} {
            \draw (1,0)--++(\th:.6);
            };
            \draw (1,0)++(45:.6) node[right] {$j$};
            \draw (1,0)++(-45:.6) node[right] {$k$};
            \filldraw[fill=white,thick] (0,0) circle(.25);
            \filldraw[fill=white,thick] (1,0) circle(.25);
            \draw (0,0) node {$0$};
            \draw (1,0) node {$0$};
        \end{tikzpicture}}$$
\end{fact}
\begin{fact}[Product of boundary divisors {\cite[Lem. 25.2.2]{HoriKatzKlemmPandharipandeThomasVafaVakilZaslow2003}}]\label{fact:ProductOfBoundary}
    Let $D_{I_1},D_{I_2}$ be boundary divisors on $\Mbar_{g,P}$ of the form in Notation~\ref{not:BasicMgnNotation}\eqref{not:BoundaryDivisor}. Then the class $D_{I_1}\cdot D_{I_2}$ is given by the following, where a half-edge decorated with $-\psi$ means the pushforward of the corresponding $\psi$-class along the inclusion of that stratum: $$\boxed{D_{I_1}\cdot D_{I_2}=\begin{cases}
        \begin{tikzpicture}[baseline=-0.65ex]
            \draw (0,0) -- (1,0);
            \draw[very thick] (0,0)--(.5,0);
            \draw (0,0) node {$\bullet$};
            \draw (1,0) node {$\bullet$};
            \foreach \th in {-60,-40,-20,0,20,40,60} {
            \draw (1,0)--++(\th:.6);
            \draw (0,0)--++(180+\th:.6);
            };
            \draw (1.6,0) node[right] {$P\setminus I_1$};
            \draw (-.6,0) node[left] {$I_1$};
            \draw (.35,0.15) node {\tiny $-\psi$};
            \filldraw[fill=white,thick] (0,0) circle(.15);
            \filldraw[fill=white,thick] (1,0) circle(.15);
            \draw (0,0) node {\tiny$0$};
            \draw (1,0) node {\tiny$g$};
        \end{tikzpicture}+\begin{tikzpicture}[baseline=-0.65ex]
            \draw (0,0) -- (1,0);
            \draw (0,0) node {$\bullet$};
            \draw (1,0) node {$\bullet$};
            \foreach \th in {-60,-40,-20,0,20,40,60} {
            \draw (1,0)--++(\th:.6);
            \draw (0,0)--++(180+\th:.6);
            };
            \draw (1.6,0) node[right] {$P\setminus I_1$};
            \draw (-.6,0) node[left] {$I_1$};
            \draw (.65,0.15) node {\tiny $-\psi$};
            \draw[very thick] (1,0)--(.5,0);
            \filldraw[fill=white,thick] (0,0) circle(.15);
            \filldraw[fill=white,thick] (1,0) circle(.15);
            \draw (0,0) node {\tiny$0$};
            \draw (1,0) node {\tiny$g$};
        \end{tikzpicture}&I_1=I_2\text{ \quad (up to $I_k\leftrightarrow P\setminus I_k$ if $g=0$)}\\
        &\\
        \begin{tikzpicture}[baseline=-0.65ex]
            \draw (0,0) -- (2,0);
            \draw (0,0) node {$\bullet$};
            \draw (1,0) node {$\bullet$};
            \foreach \th in {-60,-40,-20,0,20,40,60} {
            \draw (2,0)--++(\th:.6);
            \draw (0,0)--++(180-\th:.6);
            \draw (1,0)--++(-90+\th:.6);
            };
            \draw (2.6,0) node[right] {$P\setminus I_2$};
            \draw (1,-.6) node[below] {$I_2\setminus I_1$};
            \draw (-.6,0) node[left] {$I_1$};
            \filldraw[fill=white,thick] (0,0) circle(.25);
            \filldraw[fill=white,thick] (1,0) circle(.25);
            \filldraw[fill=white,thick] (2,0) circle(.25);
            \draw (0,0) node {$0$};
            \draw (1,0) node {$0$};
            \draw (2,0) node {$g$};
        \end{tikzpicture}&I_1\subsetneq I_2\text{ \quad (up to $I_k\leftrightarrow P\setminus I_k$ if $g=0$)}\\
        &\\
        0&\text{otherwise}.
    \end{cases}}$$
\end{fact}
In \cite{BrakensiekEurLarsonLi2023}, genus-zero intersection numbers of the form 
\begin{align}\label{eq:KapranovDegree}
    \int_{\Mbar_{0,n}}\prod_{k=1}^{n-3}\pi_{T_k}^*\psi_{i_k},
\end{align}
where for all $k=1,\ldots,n-3,$ we have $T_k\subseteq[n]$ with $\abs{T_k}\ge3$ and $i_k\in T_k$, are referred to as \emph{Kapranov degrees}. Note that the numbers $\omega_{G,0,3}$ (and in fact $\omega_{G,0,m}$) are a special case. We have the following criterion for Kapranov degrees to be nonzero (the ``Cerberus condition'' of \cite[Sec. 4]{BrakensiekEurLarsonLi2023}, see also the ``surplus-3'' condition in \cite[Sec. 4]{Silversmith2022}).
\begin{fact}[{\cite[Thm. A(b)]{BrakensiekEurLarsonLi2023}}]\label{fact:Surplus}
    The intersection number $\int_{\Mbar_{0,n}}\prod_{k=1}^{n-3}\pi_{T_k}^*\psi_{i_k}$ is nonzero if and only if, for every nonempty $K\subseteq\{1,\ldots,n-3\}$, we have $\abs{\bigcup_{k\in K}T_k}\ge\abs{K}+3$.
\end{fact}
We only need the ``only if'' direction, which is the easy one; we give its proof here.
\begin{proof}
    Suppose there exists $K$ for which $\abs{\bigcup_{k\in K}T_k}<\abs{K}+3$. The product $\prod_{k\in K}\pi_{T_k}^*\psi_{i_k}$ is the pullback of a codimension-$\abs{K}$ class on the space $\Mbar_{0,\abs{\bigcup_{k\in K}T_k}}$. Thus $\prod_{k\in K}\pi_{T_k}^*\psi_{i_k}=0$ for dimension reasons. Since the integrand $\prod_{k=1}^{n-3}\pi_{T_k}^*\psi_{i_k}$ is a multiple of $\prod_{k\in K}\pi_{T_k}^*\psi_{i_k}$, the integral vanishes.
\end{proof}
Finally, we need the following basic calculation, which is precisely the base case of Theorem \ref{thm:main} when $(g,m)=(0,3)$ (i.e. the case of graphs with no edges):
\begin{fact}\label{fact:BasicCrossRatioDegree}
    $$\int_{\Mbar_{0,n}}\pi_{1234}^*([pt])\cdot\pi_{1235}^*([pt])\cdot\cdots\cdot\pi_{123n}^*([pt])=1.$$
\end{fact}
\begin{proof}
    This is well-known and elementary, e.g. it follows from \cite[Lem. 3.2]{Silversmith2022} by induction.
\end{proof}

\section{First Proof of the main theorem: Intersection Theory of \texorpdfstring{$\Mbar_{g,n}$}{{M\_{g,n}-bar}}}\label{sec:Proof}
\begin{proof}[Proof of Theorem \ref{thm:main}]
    We proceed by induction on the number of edges in $G$. 
    
    \subsection{Base case}\label{sec:BaseCase} Let $G$ be a graph with no edges. Recall from Section \ref{sec:ProofOutline} that we need to show 
    \begin{align*}
        \omega_{G,g,m}=\begin{cases}(2g-2+m)^{\abs{V}}&(g,m,\abs{V})\ne(1,0,1)\\1&(g,m,\abs{V})=(1,0,1).\end{cases}
    \end{align*}
    First suppose $(g,m)\ne(1,0).$ Let $\Gamma$ be any $M$-marked stable genus-$g$ weighted graph\footnote{As mentioned in Remark \ref{rem:Stacks}, there are unimportant stack-theoretic subtleties in this argument, which are absent if $\Gamma$ has no automorphisms --- choosing such a $\Gamma$ is often (but not always) possible.} such that for all $\alpha\in V(\Gamma),$ $g_\alpha=0$ and $\val(\alpha)=3$, corresponding to a zero-dimensional boundary stratum in $\Mbar_{g,M}.$ Then using Notation \ref{not:BasicMgnNotation}\eqref{not:DecomposeBoundaryStratum}, we have that $\Mtilde_{\Gamma}\cong\prod_{\alpha\in V(\Gamma)}\Mbar_{0,3}$ is a single point, and $\Mbar_{\Gamma}\subseteq\Mbar_{g,M}$ is an orbifold point with isotropy group $\Aut(\Gamma)$. By Fact \ref{fact:BoundaryPullback} (and in particular \eqref{eq:PullbackOfBoundaryStrata1}), 
    \begin{align*}
        \pi_M^*([pt])=\pi_M^*(\iota_{\Gamma})_*([\Mtilde_{\Gamma}])=\sum_{\Lambda}(\iota_{\Lambda})_*([\Mtilde_{\Lambda}]),
    \end{align*} where $\Lambda$ runs over ways of attaching the marks $v\in V$ to vertices of $\Gamma$. We know that $\Gamma$ has Betti number $g,$ hence has $2g-2+m$ vertices. Each mark $v\in V$ can be attached to any of these vertices, so there are exactly $(2g-2+m)^{\abs{V}}$ graphs $\Lambda$. It thus remains to show that for each such $\Lambda$, we have \begin{align}\label{eq:BaseCaseContribution}
        \int_{\Mbar_{g,V\sqcup M}}\Psi_{G,g,m}\cdot(\iota_{\Lambda})_*([\Mtilde_{\Lambda}])=1.
    \end{align}
    Fix one such $\Lambda$. For each vertex $\alpha\in V(\Lambda)$, the set $\HE(\Lambda,\alpha)$ of half-edges based at $\alpha$ is of the form $\HE(\Lambda,\alpha)=A\sqcup B,$ where $A$ consists of half-edges in $\Gamma$, and $B$ consists of elements of $V$ that have been attached at $\alpha$. By choice of $\Gamma$, $\abs{A}=3.$ By the projection formula and Fact \ref{fact:PsiRestrict}, the left side of \eqref{eq:BaseCaseContribution} is equal to \begin{align}\label{eq:ProductOfIntegralsBaseCase}
        \prod_{\alpha\in V(\Lambda)}\left(\int_{\Mbar_{0,\HE(\Lambda,\alpha)}}\prod_{v\in B}\pi_{\{v\}\cup A}^*\psi_{v}\right).
    \end{align}
    Each integral in \eqref{eq:ProductOfIntegralsBaseCase} is precisely of the form in Fact \ref{fact:BasicCrossRatioDegree} --- the forgetful map in each factor remembers the three elements of $A\subseteq\HE(\Lambda,\alpha)$ and one other element of  $\HE(\Lambda,\alpha)$. Thus each factor of \eqref{eq:ProductOfIntegralsBaseCase} is equal to 1, so \eqref{eq:BaseCaseContribution} holds for all $\Lambda$. Thus $$\omega_{G,g,m}=(2g-2+m)^{\abs{V}}$$ as desired.

    Now suppose $(g,m)=(1,0).$ (We are still assuming $G$ has no edges.) We must show that $\omega_{G,g,m}=1$ if $G$ has a single vertex, and $\omega_{G,g,m}=0$ otherwise. If $G$ has a single vertex, then we have \begin{align}\label{eq:Psi11}
        \omega_{G,1,0}=24\cdot\int_{\Mbar_{1,1}}\psi_1=1
    \end{align} by a very well-known computation (see e.g. \cite[Cor. 1.7.2]{KockNotes}). If $G$ has more than one vertex, then by the projection formula it is sufficient to prove the identity $\int_{\Mbar_{1,2}}(\pi_1^*\psi_1)\cdot(\pi_2^*\psi_2)=0.$ This is again a standard computation. Briefly, Facts \ref{fact:PsiPullback} and \ref{fact:PsiRestrict} imply that $$(\pi_1^*\psi_1)\cdot(\pi_2^*\psi_2)=\psi_1\psi_2+D_{12}^2.$$ The first term is equal to $1/24$ by e.g. the dilaton equation \cite[Cor. 2.3.2]{KockNotes}, and the second is equal to $-1/24$ by Fact \ref{fact:ProductOfBoundary} and \eqref{eq:Psi11}. This concludes the base case.

\subsection{Inductive step}\label{sec:InductiveStep} Let $G=(V,E)$ be a graph with $\abs{E}>0,$ and let $e=\{u,w\}\in E.$ Let $V'=V(G/e)=(V\setminus\{u,w\})\sqcup\{\star\}$, where $\star\in V'$ denotes the vertex to which $e$ was contracted.
    
    For brevity, we introduce some additional notation. Let $$S:=V\sqcup M\quad\quad\text{and}\quad\quad S'=V'\sqcup M;$$ thus $\Psi_{G,g,m}$ and $\Psi_{G\setminus e,g,m}$ are cohomology classes on $\Mbar_{g,S}$, while $\Psi_{G/e,g,m}$ is a cohomology class on $\Mbar_{g,S'}$. We will repeatedly use the canonical identification $$\Mbar_{g,S'}\cong D_{uw}\xhookrightarrow{\iota_{uw}}{}\Mbar_{g,S}.$$ For all $v\in V$, let $$S_v:=N[v]\sqcup M.$$ For all $v\in V\setminus\{u,w\}$, let $$S_v'=\begin{cases}
     S_v&S_v\cap\{u,w\}=\emptyset\\
     S_v\setminus\{u,w\}\sqcup\{\star\}&S_v\cap\{u,w\}\ne\emptyset.
    \end{cases}.$$
    Finally, let $$S_\star'=((S_u\cup S_w)\setminus\{u,w\})\cup\{\star\}.$$ In this notation we have $$\Psi_{G,g,m}=\prod_{v\in V}\pi_{S_v}^*(\psi_v)\quad\quad\text{and}\quad\quad \Psi_{G/e,g,m}=\prod_{v\in V'}\pi_{S_v'}^*(\psi_v).$$
    
    In order to prove \eqref{eq:ClassDeletionContraction}, we now manipulate the expression $(\Psi_{G,g,m}-\Psi_{G\setminus e,g,m})\cdot\pi_M^*([pt])$. The products $\Psi_{G,g,m}$ and $\Psi_{G\setminus e,g,m}$ share most of the same factors, differing only in the two factors corresponding to $u$ and $w$:
\begin{align}\label{eq:FirstStep}
    \Psi_{G,g,m}-\Psi_{G\setminus e,g,m}&=\left(\prod_{v\in V\setminus\{u,w\}}\pi_{S_v}^*\psi_v\right)\cdot\left((\pi_{S_u}^*\psi_u)\cdot(\pi_{S_w}^*\psi_w)-(\pi_{S_u\setminus\{w\}}^*\psi_u)\cdot(\pi_{S_w\setminus \{u\}}^*\psi_w)\right).
\end{align}
By Fact \ref{fact:PsiPullback}, we have $$\pi_{S_u}^*\psi_u=\pi_{S_u}^*(\pi_{S_u\setminus\{w\}}^*\psi_u+D_{uw})=\pi_{S_u\setminus\{w\}}^*\psi_u+\pi_{S_u}^*D_{uw},$$ and similarly $$\pi_{S_w}^*\psi_w=\pi_{S_w}^*(\pi_{S_w\setminus\{u\}}^*\psi_w+D_{uw})=\pi_{S_w\setminus\{u\}}^*\psi_w+\pi_{S_w}^*D_{uw}.$$ Thus \eqref{eq:FirstStep} becomes:
\begin{align}
    &\Psi_{G,g,m}-\Psi_{G\setminus e,g,m}\\&=\left(\prod_{v\in V\setminus\{u,w\}}\pi_{S_v}^*\psi_v\right)\cdot\left((\pi_{S_u\setminus\{w\}}^*\psi_u+\pi_{S_u}^*D_{uw})\cdot(\pi_{S_w\setminus\{u\}}^*\psi_w+\pi_{S_w}^*D_{uw})-(\pi_{S_u\setminus\{w\}}^*\psi_u)\cdot(\pi_{S_w\setminus \{u\}}^*\psi_w)\right)\nonumber\\
    &=\left(\prod_{v\in V\setminus\{u,w\}}\pi_{S_v}^*\psi_v\right)\cdot\left((\pi_{S_u\setminus\{w\}}^*\psi_u)(\pi_{S_w}^*D_{uw})+(\pi_{S_w\setminus\{u\}}^*\psi_w)(\pi_{S_u}^*D_{uw})+(\pi_{S_u}^*D_{uw})(\pi_{S_w}^*D_{uw})\right)\nonumber
\end{align}
We therefore wish to show that $(\iota_{uw})_*(\Psi_{G/e,g,m}\cdot\pi_M^*([pt]))$ is equal to 
\begin{align}\label{eq:SimplifyFirstStep}\tag{C}
    \left(\prod_{v\in V\setminus\{u,w\}}\pi_{S_v}^*\psi_v\right)\cdot\left((\pi_{S_u\setminus\{w\}}^*\psi_u)(\pi_{S_w}^*D_{uw})+(\pi_{S_w\setminus\{u\}}^*\psi_w)(\pi_{S_u}^*D_{uw})+(\pi_{S_u}^*D_{uw})(\pi_{S_w}^*D_{uw})\right)\cdot\pi_M^*([pt]).
\end{align}
Now we're going to compute the three terms of \eqref{eq:SimplifyFirstStep} (after expanding the middle factor) using Facts \ref{fact:PsiPullback}--\ref{fact:Surplus}. The first and second terms of \eqref{eq:SimplifyFirstStep} are computed to be \eqref{eq:FirstTermContribution} and \eqref{eq:SecondTermContribution} below, respectively, whereas the third term of \eqref{eq:SimplifyFirstStep} will get expanded into four contributions \eqref{eq:ThirdTermContribution11}, \eqref{eq:ThirdTermContribution12Rewrite1}, \eqref{eq:ThirdTermContribution2}, and \eqref{eq:ThirdTermContribution3}.

\subsection*{The first term of \texorpdfstring{\eqref{eq:SimplifyFirstStep}}{Eqn C}} The first term of \eqref{eq:SimplifyFirstStep} is \begin{align}\label{eq:FirstTerm}
    \left(\prod_{v\in V\setminus\{u,w\}}\pi_{S_v}^*\psi_v\right)\cdot(\pi_{S_u\setminus\{w\}}^*\psi_u)\cdot(\pi_{S_w}^*D_{uw})\cdot\pi_M^*([pt]).
\end{align} We expand the factor $\pi_{S_w}^*D_{uw}$ using Fact \ref{fact:BoundaryPullback}, to see that \eqref{eq:FirstTerm} is equal to
\begin{align}\label{eq:FirstTermSingleTermPushforward}
    &\sum_{J\subseteq S\setminus S_w}\left(\prod_{v\in V\setminus\{u,w\}}\pi_{S_v}^*\psi_v\right)\cdot(\pi_{S_u\setminus\{w\}}^*\psi_u)\cdot\left(\begin{tikzpicture}[baseline=-0.65ex]
            \draw (0,0) -- (1,0);
            \draw (0,0) node {$\bullet$};
            \draw (1,0) node {$\bullet$};
            \draw (0,0)--++(135:.6);
            \draw (135:.6) node[left] {$u$};
            \draw (0,0)--++(180:.6);
            \draw (180:.6) node[left] {$w$};
            \foreach \x in {215,220,225,230,235} {
            \draw (0,0)--++(\x:.6);};
            \draw (225:.6) node[below left] {$J$};
            \foreach \th in {-60,-40,-20,0,20,40,60} {
            \draw (1,0)--++(\th:.6);
            };
            \draw (1.6,0) node[right] {$S\setminus(\{u,w\}\cup J)$};
            \filldraw[fill=white,thick] (0,0) circle(.25);
            \filldraw[fill=white,thick] (1,0) circle(.25);
            \draw (0,0) node {$0$};
            \draw (1,0) node {$g$};
        \end{tikzpicture}\right)\cdot\pi_M^*([pt])\\
        &\quad=\sum_{J\subseteq S\setminus S_w}(\iota_{uwJ})_*\left(\left(\prod_{v\in V\setminus\{u,w\}}\iota_{uwJ}^*\pi_{S_v}^*\psi_v\right)\cdot(\iota_{uwJ}^*\pi_{S_u\setminus\{w\}}^*\psi_u)\cdot\iota_{uwJ}^*\pi_M^*([pt])\right).\nonumber
\end{align}
For each choice of $J$, we are pushing forward from a product (see Notation \ref{not:BasicMgnNotation}\eqref{not:DecomposeBoundaryStratum}):
\begin{align}\label{eq:FirstTermProductDecomp}
\iota_{uwJ}:\Mbar_{0,\abs{J}+3}\times\Mbar_{g,\abs{V}+m-\abs{J}-2}\to\Mbar_{g,V\sqcup M}.
\end{align}
In fact, every factor inside the pushforward in \eqref{eq:FirstTermSingleTermPushforward} is itself pulled back from one of the factors of \eqref{eq:FirstTermProductDecomp}. To see this, first recall that Fact \ref{fact:PsiRestrict} expresses $\iota_{uwJ}^*\pi_{S_v}^*\psi_v$ as a forgetful-map-pullback of a $\psi$-class associated to some half-edge at a vertex of the graph in \eqref{eq:FirstTermSingleTermPushforward}. For $\iota_{uwJ}^*\pi_M^*([pt])$, we see that this class is pulled back from the second factor via the commutative diagram:
\begin{align*}
    \begin{tikzcd}[ampersand replacement=\&]
        \Mbar_{0,\abs{J}+3}\times\Mbar_{g,\abs{V}+m-\abs{J}-2}\arrow[d,"\iota_{uwJ}"]\arrow[r,"\pr_2"]\&\Mbar_{g,\abs{V}+m-\abs{J}-2}\arrow[d]\\
        \Mbar_{g,V\sqcup M}\arrow[r]\&\Mbar_{g,M}
    \end{tikzcd}
\end{align*} In other words, we can write each summand of \eqref{eq:FirstTermSingleTermPushforward} as 
\begin{align*}
    (\iota_{uwJ})_*(\pr_1^*\alpha_J\cdot\pr_2^*\beta_J),
\end{align*}
where $\alpha_J\in H^*(\Mbar_{0,\abs{J}+3})$ and $\beta_J\in H^*(\Mbar_{g,\abs{V}+m-\abs{J}-2})$. A dimension count shows that this expression is zero unless $\alpha_J$ and $\beta_J$ are zero-dimensional.

The factors $\pi_{S_v}^*\psi_v$ in \eqref{eq:FirstTermSingleTermPushforward} coming from vertices $v\in V\setminus(J\cup\{u,w\})$ will all contribute to $\beta_J.$ This is because the elements of $M$ are never forgotten, so the genus-$g$ vertex of the graph in \eqref{eq:FirstTermSingleTermPushforward} always survives stabilization.

In particular, the only factors that possibly contribute to $\alpha_J$ are the factors $\pi_{S_v}^*\psi_{v}$ for $v\in J$ and the factor $\pi_{S_u\setminus\{w\}}^*\psi_u$ --- and each of these may or may not contribute depending on whether the left vertex survives stabilization with respect to the given forgetful map. Because $J\subseteq S\setminus S_w,$ we see that for every factor that contributes to $\alpha_J$, \emph{$w$ is always forgotten.} It therefore follows from Fact \ref{fact:Surplus} that $\alpha_J=0$ unless $J=\emptyset$. (Actually, it follows that $\int_{\Mbar_{0,\abs{J}+3}}\alpha_J=0$, but this implies $\alpha_J=0$ since the integration map $H^{n-3}(\Mbar_{0,n})\to\Z$ is an isomorphism.)

Thus the first term of \eqref{eq:SimplifyFirstStep} contributes the single term:
\begin{align}\label{eq:FirstTermContribution}
    &(\iota_{uw})_*\left(\left(\prod_{v\in V\setminus\{u,w\}}\iota_{uw}^*\pi_{S_v}^*\psi_v\right)\cdot(\iota_{uw}^*\pi_{S_u\setminus\{w\}}^*\psi_u)\cdot\iota_{uw}^*\pi_M^*([pt])\right)\nonumber\\
    &\quad\quad=(\iota_{uw})_*\left(\left(\prod_{v\in V'\setminus\{\star\}}\pi_{S_v'}^*\psi_v\right)\cdot(\pi_{(S_u\cup\{\star\})\setminus\{u,w\}}^*\psi_\star)\cdot\pi_M^*([pt])\right).\tag{C1}
\end{align}

\subsection*{The second term of \texorpdfstring{\eqref{eq:SimplifyFirstStep}}{Eqn C}} By an identical calculation, the second term of \eqref{eq:SimplifyFirstStep} is equal to the pushforward: \begin{align}\label{eq:SecondTermContribution}\tag{C2}
    (\iota_{uw})_*\left(\left(\prod_{v\in V'\setminus\{\star\}}\pi_{S_v'}^*\psi_v\right)\cdot(\pi_{(S_w\cup\{\star\})\setminus\{u,w\}}^*\psi_\star)\cdot\pi_M^*([pt])\right).
\end{align}

\subsection*{The third term of \texorpdfstring{\eqref{eq:SimplifyFirstStep}}{Eqn C}} The third term of \eqref{eq:SimplifyFirstStep} is \begin{align}\label{eq:ThirdTerm}
    \left(\prod_{v\in V\setminus\{u,w\}}\pi_{S_v}^*\psi_v\right)\cdot(\pi_{S_u}^*D_{uw})\cdot(\pi_{S_w}^*D_{uw})\cdot\pi_M^*([pt]).
\end{align}
By Fact \ref{fact:BoundaryPullback}, \eqref{eq:ThirdTerm} is equal to 
\begin{align}\label{eq:ThirdTermSimplify}
    \left(\prod_{v\in V\setminus\{u,w\}}\pi_{S_v}^*\psi_v\right)\cdot\left(\begin{tikzpicture}[baseline=-0.65ex]
            \draw (0,0) -- (1,0);
            \draw (0,0) node {$\bullet$};
            \draw (1,0) node {$\bullet$};
            \draw (0,0)--++(135:.6);
            \draw (135:.6) node[left] {$u$};
            \draw (0,0)--++(225:.6);
            \draw (225:.6) node[left] {$w$};
            \foreach \th in {-60,-40,-20,0,20,40,60} {
            \draw (1,0)--++(\th:.6);
            };
            \draw (1.6,0) node[right] {$S_u\setminus(\{u,w\})$};
            \filldraw[fill=white,thick] (0,0) circle(.25);
            \filldraw[fill=white,thick] (1,0) circle(.25);
            \draw (0,0) node {$0$};
            \draw (1,0) node {$g$};
        \end{tikzpicture}\right)\cdot\left(\begin{tikzpicture}[baseline=-0.65ex]
            \draw (0,0) -- (1,0);
            \draw (0,0) node {$\bullet$};
            \draw (1,0) node {$\bullet$};
            \draw (0,0)--++(135:.6);
            \draw (135:.6) node[left] {$u$};
            \draw (0,0)--++(225:.6);
            \draw (225:.6) node[left] {$w$};
            \foreach \th in {-60,-40,-20,0,20,40,60} {
            \draw (1,0)--++(\th:.6);
            };
            \draw (1.6,0) node[right] {$S_w\setminus(\{u,w\})$};
            \filldraw[fill=white,thick] (0,0) circle(.25);
            \filldraw[fill=white,thick] (1,0) circle(.25);
            \draw (0,0) node {$0$};
            \draw (1,0) node {$g$};
        \end{tikzpicture}\right)\cdot\pi_M^*([pt]),
\end{align}
where recall that we are using the shorthand e.g.$$\left(\begin{tikzpicture}[baseline=-0.65ex]
            \draw (0,0) -- (1,0);
            \draw (0,0) node {$\bullet$};
            \draw (1,0) node {$\bullet$};
            \draw (0,0)--++(135:.6);
            \draw (135:.6) node[left] {$u$};
            \draw (0,0)--++(225:.6);
            \draw (225:.6) node[left] {$w$};
            \foreach \th in {-60,-40,-20,0,20,40,60} {
            \draw (1,0)--++(\th:.6);
            };
            \draw (1.6,0) node[right] {$S_u\setminus(\{u,w\})$};
            \filldraw[fill=white,thick] (0,0) circle(.25);
            \filldraw[fill=white,thick] (1,0) circle(.25);
            \draw (0,0) node {$0$};
            \draw (1,0) node {$g$};
        \end{tikzpicture}\right)=\sum_{J_1\subseteq S\setminus S_u}\left(\begin{tikzpicture}[baseline=-0.65ex]
            \draw (0,0) -- (1,0);
            \draw (0,0) node {$\bullet$};
            \draw (1,0) node {$\bullet$};
            \draw (0,0)--++(135:.6);
            \draw (135:.6) node[left] {$u$};
            \draw (0,0)--++(180:.6);
            \draw (180:.6) node[left] {$w$};
            \foreach \x in {215,220,225,230,235} {
            \draw (0,0)--++(\x:.6);};
            \draw (225:.6) node[below left] {$J_1$};
            \foreach \th in {-60,-40,-20,0,20,40,60} {
            \draw (1,0)--++(\th:.6);
            };
            \draw (1.6,0) node[right] {$S\setminus(\{u,w\}\cup J_1)$};
            \filldraw[fill=white,thick] (0,0) circle(.25);
            \filldraw[fill=white,thick] (1,0) circle(.25);
            \draw (0,0) node {$0$};
            \draw (1,0) node {$g$};
        \end{tikzpicture}\right).$$

Similarly the second boundary stratum in \eqref{eq:ThirdTermSimplify} is a sum over $J_2\subseteq S\setminus S_w.$ We now distribute these two sums --- by Fact \ref{fact:ProductOfBoundary} there are three types of nonzero contribution, corresponding to the cases $J_1=J_2$, $J_1\subsetneq J_2$, and $J_2\subsetneq J_1$:
\begin{enumerate}
    \item From the case $J_1=J_2$, we get \begin{align*}\pi_M^*([pt])\cdot\left(\prod_{v\in V\setminus\{u,w\}}\pi_{S_v}^*\psi_v\right)\cdot\sum_{J\subseteq S\setminus(S_u\cup S_w)}&\left({\begin{tikzpicture}[baseline=-0.65ex]
            \draw (0,0) -- (1,0);
            \draw (0,0) node {$\bullet$};
            \draw (1,0) node {$\bullet$};
            \draw (0,0) node[above right] {\tiny $-\psi$};
            \draw (0,0)--++(135:.6);
            \draw (135:.6) node[left] {$u$};
            \draw (0,0)--++(180:.6);
            \draw (180:.6) node[left] {$w$};
            \foreach \x in {215,220,225,230,235} {
            \draw (0,0)--++(\x:.6);};
            \draw (225:.6) node[below left] {$J$};
            \foreach \th in {-60,-40,-20,0,20,40,60} {
            \draw (1,0)--++(\th:.6);
            };
            \draw (1.6,0) node[right] {$S\setminus(\{u,w\}\cup J)$};
            \draw[very thick] (0,0)--(.5,0);
            \filldraw[fill=white,thick] (0,0) circle(.15);
            \filldraw[fill=white,thick] (1,0) circle(.15);
            \draw (0,0) node {\tiny$0$};
            \draw (1,0) node {\tiny$g$};
        \end{tikzpicture}}+\right.\\&\quad\quad\quad\quad\left.{\begin{tikzpicture}[baseline=-0.65ex]
            \draw (0,0) -- (1,0);
            \draw (0,0) node {$\bullet$};
            \draw (1,0) node {$\bullet$};
            \draw (1,0) node[above left] {\tiny $-\psi$};
            \draw (0,0)--++(135:.6);
            \draw (135:.6) node[left] {$u$};
            \draw (0,0)--++(180:.6);
            \draw (180:.6) node[left] {$w$};
            \foreach \x in {215,220,225,230,235} {
            \draw (0,0)--++(\x:.6);};
            \draw (225:.6) node[below left] {$J$};
            \foreach \th in {-60,-40,-20,0,20,40,60} {
            \draw (1,0)--++(\th:.6);
            };
            \draw (1.6,0) node[right] {$S\setminus(\{u,w\}\cup J)$};
            \draw[very thick] (1,0)--(.5,0);
            \filldraw[fill=white,thick] (0,0) circle(.15);
            \filldraw[fill=white,thick] (1,0) circle(.15);
            \draw (0,0) node {\tiny$0$};
            \draw (1,0) node {\tiny$g$};
        \end{tikzpicture}}\right).\end{align*}
        For the term where the $-\psi$ is on the genus-$g$ vertex, we use the same argument as we did in the analysis of the first term of \eqref{eq:SimplifyFirstStep} above --- applying Fact \ref{fact:Surplus} to the genus-zero vertex --- to conclude that we only have a nonzero contribution when $J=\emptyset$. 
        Thus we get the two contributions
        \begin{align}\label{eq:ThirdTermContribution11}
    \pi_M^*([pt])\cdot\left(\prod_{v\in V\setminus\{u,w\}}\pi_{S_v}^*\psi_v\right)\cdot&\left({\begin{tikzpicture}[baseline=-0.65ex]
            \draw (0,0) -- (1,0);
            \draw (0,0) node {$\bullet$};
            \draw (1,0) node {$\bullet$};
            \draw (1,0) node[above left] {\tiny $-\psi$};
            \draw (0,0)--++(135:.6);
            \draw (135:.6) node[left] {$u$};
            \draw (0,0)--++(225:.6);
            \draw (225:.6) node[left] {$w$};
            \foreach \th in {-60,-40,-20,0,20,40,60} {
            \draw (1,0)--++(\th:.6);
            };
            \draw (1.6,0) node[right] {$S\setminus(\{u,w\})$};
            \draw[very thick] (1,0)--(.5,0);
            \filldraw[fill=white,thick] (0,0) circle(.15);
            \filldraw[fill=white,thick] (1,0) circle(.15);
            \draw (0,0) node {\tiny$0$};
            \draw (1,0) node {\tiny$g$};
        \end{tikzpicture}}\right)\nonumber\\&=-(\iota_{uw})_*\left(\pi_M^*([pt])\cdot\psi_\star\cdot\prod_{v\in V'\setminus\{\star\}}\pi_{S_v'}^*\psi_v\right),\tag{C3}
\end{align}
where we have again used Fact \ref{fact:PsiRestrict},
and \begin{align}\label{eq:ThirdTermContribution12}
    \pi_M^*([pt])\cdot\left(\prod_{v\in V\setminus\{u,w\}}\pi_{S_v}^*\psi_v\right)\cdot\sum_{J\subseteq S\setminus(S_u\cup S_w)}\left(\begin{tikzpicture}[baseline=-0.65ex]
            \draw (0,0) -- (1,0);
            \draw (0,0) node {$\bullet$};
            \draw (1,0) node {$\bullet$};
            \draw (0,0) node[above right] {\tiny $-\psi$};
            \draw (0,0)--++(135:.6);
            \draw (135:.6) node[left] {$u$};
            \draw (0,0)--++(180:.6);
            \draw (180:.6) node[left] {$w$};
            \foreach \x in {215,220,225,230,235} {
            \draw (0,0)--++(\x:.6);};
            \draw (225:.6) node[left] {$J$};
            \foreach \th in {-60,-40,-20,0,20,40,60} {
            \draw (1,0)--++(\th:.6);
            };
            \draw (1.6,0) node[right] {$S\setminus(\{u,w\}\cup J)$};
            \draw[very thick] (0,0)--(.5,0);
            \filldraw[fill=white,thick] (0,0) circle(.15);
            \filldraw[fill=white,thick] (1,0) circle(.15);
            \draw (0,0) node {\tiny$0$};
            \draw (1,0) node {\tiny$g$};
        \end{tikzpicture}\right).
\end{align}
Applying Fact \ref{fact:WritePsiAsBoundary} to the left vertex, we may rewrite \eqref{eq:ThirdTermContribution12} as:
\begin{align}\label{eq:ThirdTermContribution12Rewrite1}
    -\pi_M^*([pt])\cdot&\left(\prod_{v\in V\setminus\{u,w\}}\pi_{S_v}^*\psi_v\right)\cdot\sum_{\substack{J\subseteq S\setminus(S_u\cup S_w)\\J'\sqcup J''=J\\
    J''\ne\emptyset}}\left({\begin{tikzpicture}[baseline=-0.65ex]
            \draw (0,0) -- (2,0);
            \draw (0,0) node {$\bullet$};
            \draw (1,0) node {$\bullet$};
            \draw (2,0) node {$\bullet$};
            \draw (0,0)--++(135:.6);
            \draw (135:.6) node[left] {$u$};
            \draw (0,0)--++(180:.6);
            \draw (180:.6) node[left] {$w$};
            \foreach \x in {215,220,225,230,235} {
            \draw (0,0)--++(\x:.6);};
            \draw (225:.6) node[left] {$J'$};
            \foreach \x in {-80,-85,-90,-95,-100} {
            \draw (1,0)--++(\x:.6);};
            \draw (1,0)++(-90:.6) node[below] {$J''$};
            \foreach \th in {-60,-40,-20,0,20,40,60} {
            \draw (2,0)--++(\th:.6);
            };
            \draw (2.6,0) node[right] {$S\setminus(\{u,w\}\cup J)$};
            \filldraw[fill=white,thick] (0,0) circle(.25);
            \filldraw[fill=white,thick] (1,0) circle(.25);
            \filldraw[fill=white,thick] (2,0) circle(.25);
            \draw (0,0) node {$0$};
            \draw (1,0) node {$0$};
            \draw (2,0) node {$g$};
        \end{tikzpicture}}\right)\nonumber\\
        &=-\pi_M^*([pt])\cdot\left(\prod_{v\in V\setminus\{u,w\}}\pi_{S_v}^*\psi_v\right)\cdot\sum_{\substack{J\subseteq S\setminus(S_u\cup S_w)\\J\ne\emptyset}}\left({\begin{tikzpicture}[baseline=-0.65ex]
            \draw (0,0) -- (2,0);
            \draw (0,0) node {$\bullet$};
            \draw (1,0) node {$\bullet$};
            \draw (2,0) node {$\bullet$};
            \draw (0,0)--++(135:.6);
            \draw (135:.6) node[left] {$u$};
            \draw (0,0)--++(225:.6);
            \draw (225:.6) node[left] {$w$};
            \foreach \x in {-80,-85,-90,-95,-100} {
            \draw (1,0)--++(\x:.6);};
            \draw (1,0)++(-90:.6) node[below] {$J$};
            \foreach \th in {-60,-40,-20,0,20,40,60} {
            \draw (2,0)--++(\th:.6);
            };
            \draw (2.6,0) node[right] {$S\setminus(\{u,w\}\cup J)$};
            \filldraw[fill=white,thick] (0,0) circle(.25);
            \filldraw[fill=white,thick] (1,0) circle(.25);
            \filldraw[fill=white,thick] (2,0) circle(.25);
            \draw (0,0) node {$0$};
            \draw (1,0) node {$0$};
            \draw (2,0) node {$g$};
        \end{tikzpicture}}\right)\nonumber\\
        &=-(\iota_{uw})_*\left(\pi_M^*([pt])\cdot\left(\prod_{v\in V'\setminus\{\star\}}\pi_{S_v'}^*\psi_v\right)\cdot\left({\begin{tikzpicture}[baseline=-0.65ex]
            \draw (1,0) -- (2,0);
            \draw (1,0) node {$\bullet$};
            \draw (2,0) node {$\bullet$};
            \draw (1,0)--++(135:.6);
            \draw (1,0)++(135:.6) node[left] {$\star$};
            \foreach \th in {-60,-40,-20,0,20,40,60} {
            \draw (2,0)--++(\th:.6);
            };
            \draw (2.6,0) node[right] {$S_\star'\setminus\{\star\}$};
            \filldraw[fill=white,thick] (1,0) circle(.25);
            \filldraw[fill=white,thick] (2,0) circle(.25);
            \draw (1,0) node {$0$};
            \draw (2,0) node {$g$};
        \end{tikzpicture}}\right)\right).\tag{C4}
\end{align}
where we have yet again applied Fact \ref{fact:Surplus} to the left vertex (where we have a Kapranov degree that excludes $u$ and $w$) to conclude $J'=\emptyset$. 
        
        \bigskip
        
        \item From the case $J_1\subsetneq J_2$, we get 
        \begin{align}\label{eq:ThirdTermContribution2}
        \pi_M^*([pt])\cdot&\left(\prod_{v\in V\setminus\{u,w\}}\pi_{S_v}^*\psi_v\right)\cdot\sum_{\substack{J_1\subset S\setminus S_u\\J_2\subset S\setminus S_w\\J_1\subsetneq J_2}}\left(\begin{tikzpicture}[baseline=-0.65ex]
            \draw (0,0) -- (2,0);
            \draw (0,0) node {$\bullet$};
            \draw (2,0) node {$\bullet$};
            \draw (1,0) node {$\bullet$};
            \draw (0,0)--++(135:.6);
            \draw (135:.6) node[left] {$u$};
            \draw (0,0)--++(180:.6);
            \draw (180:.6) node[left] {$w$};
            \foreach \x in {215,220,225,230,235} {
            \draw (0,0)--++(\x:.6);};
            \draw (225:.6) node[below left] {$J_1$};
            \foreach \x in {-80,-85,-90,-95,-100} {
            \draw (1,0)--++(\x:.6);};
            \draw (1,0)++(-90:.6) node[below] {$J_2\setminus J_1$};
            \foreach \th in {-60,-40,-20,0,20,40,60} {
            \draw (2,0)--++(\th:.6);
            };
            \draw (2.6,0) node[right] {$S\setminus(\{u,w\}\cup J_2)$};
            \filldraw[fill=white,thick] (0,0) circle(.25);
            \filldraw[fill=white,thick] (1,0) circle(.25);
            \filldraw[fill=white,thick] (2,0) circle(.25);
            \draw (0,0) node {$0$};
            \draw (1,0) node {$0$};
            \draw (2,0) node {$g$};
        \end{tikzpicture}\right)\nonumber\\
        &=\pi_M^*([pt])\cdot\left(\prod_{v\in V\setminus\{u,w\}}\pi_{S_v}^*\psi_v\right)\cdot\sum_{\substack{J_2\subset S\setminus S_w\\J_2\ne\emptyset}}\left(\begin{tikzpicture}[baseline=-0.65ex]
            \draw (0,0) -- (2,0);
            \draw (0,0) node {$\bullet$};
            \draw (2,0) node {$\bullet$};
            \draw (1,0) node {$\bullet$};
            \draw (0,0)--++(135:.6);
            \draw (135:.6) node[left] {$u$};
            \draw (0,0)--++(225:.6);
            \draw (225:.6) node[left] {$w$};
            \foreach \x in {-80,-85,-90,-95,-100} {
            \draw (1,0)--++(\x:.6);};
            \draw (1,0)++(-90:.6) node[below] {$J_2$};
            \foreach \th in {-60,-40,-20,0,20,40,60} {
            \draw (2,0)--++(\th:.6);
            };
            \draw (2.6,0) node[right] {$S\setminus(\{u,w\}\cup J_2)$};
            \filldraw[fill=white,thick] (0,0) circle(.25);
            \filldraw[fill=white,thick] (1,0) circle(.25);
            \filldraw[fill=white,thick] (2,0) circle(.25);
            \draw (0,0) node {$0$};
            \draw (1,0) node {$0$};
            \draw (2,0) node {$g$};
        \end{tikzpicture}\right)\nonumber\\
        &=(\iota_{uw})_*\left(\pi_M^*([pt])\cdot\left(\prod_{v\in V'\setminus\{\star\}}\pi_{S_v'}^*\psi_v\right)\cdot\left(\begin{tikzpicture}[baseline=-0.65ex]
            \draw (1,0) -- (2,0);
            \draw (2,0) node {$\bullet$};
            \draw (1,0) node {$\bullet$};
            \draw (1,0)--++(135:.6);
            \draw (1,0)++(135:.6) node[left] {$\star$};
            \foreach \th in {-60,-40,-20,0,20,40,60} {
            \draw (2,0)--++(\th:.6);
            };
            \draw (2.6,0) node[right]{$S_w\setminus\{u,w\}$};
            \filldraw[fill=white,thick] (1,0) circle(.25);
            \filldraw[fill=white,thick] (2,0) circle(.25);
            \draw (1,0) node {$0$};
            \draw (2,0) node {$g$};
        \end{tikzpicture}\right)\right).\tag{C5}
        \end{align}
        Here we again have used Fact \ref{fact:Surplus} to conclude $J_1=\emptyset,$ since otherwise the Kapranov degree on the left vertex excludes $u$ by the fact $J_1\cap S_u=\emptyset.$ 
\item By an identical calculation, from the case $J_1\subsetneq J_2$ we get the contribution
\begin{align}\label{eq:ThirdTermContribution3}\tag{C6}
    (\iota_{uw})_*\left(\pi_M^*([pt])\cdot\left(\prod_{v\in V'\setminus\{\star\}}\pi_{S_v'}^*\psi_v\right)\cdot\left(\begin{tikzpicture}[baseline=-0.65ex]
            \draw (1,0) -- (2,0);
            \draw (2,0) node {$\bullet$};
            \draw (1,0) node {$\bullet$};
            \draw (1,0)--++(135:.6);
            \draw (1,0)++(135:.6) node[left] {$\star$};
            \foreach \th in {-60,-40,-20,0,20,40,60} {
            \draw (2,0)--++(\th:.6);
            };
            \draw (2.6,0) node[right] {$S_u\setminus\{u,w\}$};
            \filldraw[fill=white,thick] (1,0) circle(.25);
            \filldraw[fill=white,thick] (2,0) circle(.25);
            \draw (1,0) node {$0$};
            \draw (2,0) node {$g$};
        \end{tikzpicture}\right)\right).
\end{align}
\end{enumerate}
\subsection*{Adding the contributions \texorpdfstring{$\eqref{eq:FirstTermContribution},\ldots,\eqref{eq:ThirdTermContribution3}$}{C1 up to C6}}
Note that the classes $\eqref{eq:FirstTermContribution},\ldots,\eqref{eq:ThirdTermContribution3}$ are all pushforwards along $\iota_{uw}$ (as originally desired!), and furthermore they are pushforwards of multiples of the product $\prod_{v\in V'\setminus\{\star\}}\pi_{S_v'}^*\psi_v.$ Observe that by Fact \ref{fact:PsiPullback}, we have
\begin{align*}
    \eqref{eq:FirstTermContribution}+\eqref{eq:ThirdTermContribution3}=\eqref{eq:SecondTermContribution}+\eqref{eq:ThirdTermContribution2}=-\eqref{eq:ThirdTermContribution11}
\end{align*}
Thus 
\begin{align*}
    (\Psi_{G,g,m}-\Psi_{G\setminus e,g,m})\cdot\pi_M^*([pt])&=\eqref{eq:FirstTermContribution}+\eqref{eq:SecondTermContribution}+\eqref{eq:ThirdTermContribution11}+\eqref{eq:ThirdTermContribution12Rewrite1}+\eqref{eq:ThirdTermContribution2}+\eqref{eq:ThirdTermContribution3}\\
    &=\eqref{eq:ThirdTermContribution12Rewrite1}-\eqref{eq:ThirdTermContribution11}\\
    &=(\iota_{uw})_*\left(\pi_M^*([pt])\cdot\left(\prod_{v\in V'\setminus\{\star\}}\pi_{S_v'}^*\psi_v\right)\cdot\left(\psi_\star-\left({\begin{tikzpicture}[baseline=-0.65ex]
            \draw (1,0) -- (2,0);
            \draw (1,0) node {$\bullet$};
            \draw (2,0) node {$\bullet$};
            \draw (1,0)--++(135:.6);
            \draw (1,0)++(135:.6) node[left] {$\star$};
            \foreach \th in {-60,-40,-20,0,20,40,60} {
            \draw (2,0)--++(\th:.6);
            };
            \draw (2.6,0) node[right] {$S_\star'\setminus\{\star\}$};
            \filldraw[fill=white,thick] (1,0) circle(.25);
            \filldraw[fill=white,thick] (2,0) circle(.25);
            \draw (1,0) node {$0$};
            \draw (2,0) node {$g$};
        \end{tikzpicture}}\right)\right)\right)\\
        &=(\iota_{uw})_*\left(\pi_M^*([pt])\cdot\left(\prod_{v\in V'\setminus\{\star\}}\pi_{S_v'}^*\psi_v\right)\cdot\pi_{S_\star'}^*\psi_\star\right)\hspace{1in}\text{(by Fact \ref{fact:PsiPullback})}\\
        &=(\iota_{uw})_*\left(\pi_M^*([pt])\cdot\prod_{v\in V'}\pi_{S_v'}^*\psi_v\right)\\
        &=(\iota_{uw})_*\left(\pi_M^*([pt])\cdot\Psi_{G/e,g,m}\right).
\end{align*}
Integrating the above yields
\begin{align}\label{eq:FinalDeletionContraction}
    \omega_{G,g,m}=\omega_{G\setminus e,g,m}+\omega_{G/e,g,m}.
\end{align}
We now complete the induction argument. Note that $G\setminus e$ has $\abs{V}$ vertices and $G/e$ has $\abs{V}-1$ vertices, while both have strictly fewer edges than $G$. By the inductive hypothesis, we have
\begin{align*}
    \omega_{G\setminus e,g,m}&=\begin{cases}
        (-1)^{\abs{V}}\chi_{G\setminus e}(-(2g-2+m))&(g,m)\ne(1,0)\\
        (-1)^{\abs{V}-1}\left.\deriv{}{x}\chi_{G\setminus e}(x)\right|_{x=0}&(g,m)=(1,0)
    \end{cases}
    \end{align*}
    and
    \begin{align*}
        \omega_{G/e,g,m}&=\begin{cases}
        (-1)^{\abs{V}-1}\chi_{G/e}(-(2g-2+m))&(g,m)\ne(1,0)\\
        (-1)^{\abs{V}-2}\left.\deriv{}{x}\chi_{G/e}(x)\right|_{x=0}&(g,m)=(1,0).
    \end{cases}
\end{align*}
Plugging these into the right side of \eqref{eq:FinalDeletionContraction} and applying the deletion-contraction formula \eqref{eq:DeletionContractionChromatic} for chromatic polynomials completes the proof of Theorem \ref{thm:main}.
\end{proof}

\section{Second Proof of the main theorem: Hyperplane Arrangements and Critical Points}
\label{sec:hyperplanes}

In this section we prove Theorems~\ref{thm:num_bounded_components} and \ref{thm:critical_point_interpretation}. This will complete the second proof of Theorem \ref{thm:main} given in Section \ref{sec:MLDegrees}.
In this section, we fix a finite simple graph $G=(V,E)$, and an integer $m\ge3$. Let $\mathcal{A}_{\R}(G;m)$ and $\mathcal{A}_{\C}(G;m)$ be as in Section \ref{sec:MLDegrees}.

Let us set up some notation: for
$v \in V$ and $i \in \set{0, \dots, m-2}$, let $A_{v,i}$ be the hyperplane $\set{\boldz \in \C^V \colon z_v = i}$ and for $e = (v,w) \in E$, let $H_e$ denote the hyperplane $\set{\boldz \in \C^V \colon z_v = z_w}$. For a hyperplane $H$ of either type, we choose a defining equation $f_H$.
We also write $u_{v,i}$ for $u_{A_{v,i}}$ and $u_e$ for $u_{H_e}$. 

The likelihood function from Section \ref{sec:MLDegrees} is given by:
\begin{align}\label{eq:LikelihoodFunctionG}
    F_{\calA(G;m), \boldu}=\prod_{H\in\calA(G;m)}f_H^{u_H}.
\end{align}
The logarithmic derivatives are:
\begin{equation}
    \frac{\partial \log F_{\calA(G;m), \boldu}}{\partial z_v} = \sum_{H \in \calA(G;m)} \frac{u_H}{f_H} \frac{\partial f_H}{\partial z_v} = \sum^{m-2}_{i=0} \frac{u_{v,i}}{z_v - i} + \sum_{e = (v,w) \in E} \frac{u_e}{z_v - z_w}
    \label{eqn:log_diff}
\end{equation}
\subsection{Counting bounded regions and critical points}
 We first prove Theorem \ref{thm:num_bounded_components} from Section \ref{sec:MLDegrees}:

\medskip

\noindent\textbf{Theorem \ref{thm:num_bounded_components} .}
    $\mathcal{A}_{\R}(G;m)$ has exactly $(-1)^{\abs{V}}\chi_G(-(m-2))$ bounded regions.

\begin{proof}[Proof of Theorem~\ref{thm:num_bounded_components}]
    We give a bijection from bounded regions of $\calA_{\R}(G;m)$ to acyclic orientations of $G$ together with a compatible coloring. First note that if $\boldx=(x_v)_{v\in V} \in \R^V \setminus \calA_{\R}(G;m)$ is in a bounded region, then $$0 < \min_{v\in V} x_v \leq \max_{v\in V} x_v < m-2,$$ as otherwise there is a ray in the region that only changes the minimal or the maximal coordinate, respectively.

    So every bounded component is contained in $[0,m-2]^V$.
    Let $\boldx \in [0,m-2]^V \setminus \calA_{\R}(G;m)$. We associate to it an orientation of $G$ as follows: if $e = \set{u,w} \in E$, then $x_u \not= x_w$ as $\boldx \notin \calA_{\R}(G;m)$, so we orient the edge $u\to w$ if $x_u > x_w$, and $w\to u$ otherwise. To every vertex, we associate the color $\sigma(v) = \ceil{x_v}$. This datum defines an acyclic orientation of $G$ together with a compatible coloring.
    
    We note this map defines a bijection between bounded components of $\R^V \setminus \calA_{\R}(G;m)$ and acyclic orientations of $G$ together with a compatible coloring in $\{1, \dots, m - 2\}$. By Theorem~\ref{thm:Stanley}, the number of pairs is $(-1)^{|V|} \chi_G( - (m - 2))$.
\end{proof}
\begin{cor}\label{cor:CountCriticalPoints}
    For general $\boldu\in\C^{\#\mathcal{A}(G;m)}$, the function $F_{\mathcal{A}(G;m),\boldu}$ has exactly $(-1)^{\abs{V}}\chi_G(-(m-2))$ critical points on its domain $\C^V\setminus\calA_{\C}(G;m)$, all of which are non-degenerate.
\end{cor}
\begin{proof}
    Non-degeneracy follows immediately from Theorem \ref{thm:varschenko}\eqref{it:FinitelyManyCriticalPoints}.
    
    By Theorem \ref{thm:num_bounded_components}, $\mathcal{A}_{\R}(G;m)$ has exactly $(-1)^{\abs{V}}\chi_G(-(m-2))$ bounded components. By Theorem~\ref{thm:varschenko}\eqref{it:1PerRegion}, for general $\boldu\in\R_{>0}^{\#\mathcal{A}(G;m)}$, $F_{\mathcal{A}(G;m),\boldu}$ has exactly $(-1)^{\abs{V}}\chi_G(-(m-2))$ critical points. By Theorem~\ref{thm:varschenko}(\ref{it:FinitelyManyCriticalPoints}--\ref{it:CountCriticalPoints}), there is a Zariski open subset $U\subseteq\C^{\#\calA(G;m)}$ such that for $\boldu\in U$, the number of critical points of $F_{\mathcal{A}(G;m),\boldu}$ is constant. Since $\R_{>0}^{\#\mathcal{A}(G;m)}$ is Zariski dense in $\C^{\#\mathcal{A}(G;m)}$, we conclude that for $\boldu\in U$, $F_{\mathcal{A}(G;m),\boldu}$ has exactly $(-1)^{\abs{V}}\chi_G(-(m-2))$ critical points.
\end{proof}
In fact, choosing $\boldu$ ``more generally'' (i.e. from a smaller subset $U$), we get the following slightly stronger statement.
\begin{lemma}\label{lem:CritAvoidsKV}
    For general $\boldu\in\C^{\#\mathcal{A}(G;m)}$, the critical points of $F_{\mathcal{A}(G;m),\boldu}$ all lie in the open subset $$\C^V\setminus\calA_{\C}(K_V;m)\subseteq\C^V\setminus\calA_{\C}(G;m),$$ where $K_V$ denotes the complete graph with vertex set $V$.
 \end{lemma}
 \begin{proof}
    We note that $F_{\calA(G;m), \boldu}$ and $F_{\calA(G;m), \lambda \boldu}$ for $\lambda \in \C^*$ have the same critical points, as \eqref{eqn:log_diff} is linear in $\boldu$. We can consider the following set
    \begin{equation*}
        \frakX^\circ := \set{ (\boldz, \boldu) \in \C^V \setminus \calA(G; m) \times \P^{\# \calA(G; m) - 1}  : \boldz \text{ is a critical point of } F_{\calA(G;m), \boldu} }.
    \end{equation*}
    This is called the (open) \emph{likelihood correspondence}, compare \cite[Section 2.2]{Huh2013} or \cite[Theorem 1.6]{HuhSturmfels}. The projection $\pr_1 : \frakX^\circ \rightarrow \C^V \setminus \calA(G; m)$ is a projective bundle of total dimension $\# \calA(G; m) - 1$, in particular the map is surjective, and as $\C^V \setminus \calA(G; m)$ is irreducible, so is $\frakX^\circ$.
    In particular, if $Y \subsetneq \C^V \setminus \calA(G; m)$ is Zariski closed, then $\pr_1^{-1}(Y)$ has positive codimension in $\frakX^\circ$. By dimension considerations, $\pr_2 (\pr_1^{-1}(Y))$ is contained in a proper subvariety of $\P^{\# \calA(G; m) - 1}$. Applying this to $Y=\calA(K_V; m)\setminus\calA(G; m)$, we conclude that for $\boldu \in \C^{\# \calA(G; m)}$ generic, $F_{\calA(G;m), \boldu}$ has no critical points in $\calA(K_V; m)$.
\end{proof}

\subsection{Comparing intersections}
In order to prove Theorem~\ref{thm:critical_point_interpretation}, we must relate the critical points of $F_{\calA(G;m), \boldu}$ to the intersection product on $\Mbar_{0, V \sqcup M}$ that defines $\omega_{G,0,m}$. We now describe this intersection product; the relationship to \eqref{eqn:log_diff} is Theorem \ref{thm:IntersectExplicitHypersurfaces}.

Let $y\in\M_{0, M}$ be the configuration $(0, 1, \dots, m - 2, \infty)$ of points on $\P^1$. Then we have a natural isomorphism 
\begin{align}\label{eq:IdentifyCVWithFiber}
    \C^{V}\setminus\calA_{\C}(K_V;m)\cong\pi_M^{-1}(y) \cap \M_{0,V \sqcup M}\subseteq\Mbar_{0,V\sqcup M},
\end{align}
sending a tuple $(z_v)_{v\in V}\in\C^{V}\setminus\calA_{\C}(K_V;m)$ to the configuration $((z_v)_{v\in V},0,1,\ldots,m-2,\infty)$ of points on $\P^1$.

Next, recall the \emph{Kapranov morphism} \cite{Kapranov1993}.
\begin{lemma}
\label{lem:kapranov_morphism}
    Let $S = \set{v, w_1, \ldots , w_k, p_\infty}$ be a finite set of cardinality $k + 2$.
    Then the complete linear system of the divisor class $\psi_v$ gives a birational morphism $|\psi_v| : \Mbar_{0,S} \rightarrow \P^{k-1}$ called the Kapranov morphism, which can be written in coordinates as follows.

    Let $C\in\Mbar_{0,S}$, and let $C'\subseteq C$ be the irreducible component containing $v$. Contracting all components except $C'$ gives a map $f:C\to C'.$ We choose coordinates $C'\cong\P^1$ sending $f(p_\infty)\mapsto\infty.$ Each element $s\in S$ is then identified with a point $z_s\in\P^1.$ 
    Then \begin{align*}
            \abs{\psi_v}(C) = \left[\frac{1}{z_v - z_{w_1}} : \quad \cdots\quad : \frac{1}{z_v - z_{w_k}}\right]\in\P^{k-1},
        \end{align*}
    with the convention that $\frac{1}{z_v-z_{w_i}}=0$ if $z_{w_i}=\infty$. (Note $z_{w_i}\ne z_v$ for all $i=1,\ldots,k$.)
\end{lemma}
We define, for $\boldu\in (\C^*)^{\calA(G;m)}$, an explicit collection of hypersurfaces $L_{v,\boldu}\subseteq\Mbar_{0,V\sqcup M}$ for all $v\in V$. Let $l_{v,\boldu}\subseteq\P^{\val(v)+m-1}$ be defined by the linear equation 
    \begin{align}\label{eq:HyperplaneEquation}
        \sum_{e=(v,w)}u_{e}x_w+\sum_{i=0}^{m-2}u_{v,i}x_i=0.
    \end{align}
    We then let $L_{v,\boldu}=(\abs{\psi_v}\circ\pi_{N[v]\cup M})^{-1}(l_{v,\boldu}),$ where $\abs{\psi_v}$ is the Kapranov morphism in the coordinates of Lemma \ref{lem:kapranov_morphism}.
    
    The following theorem is what will allow us to match up the computation of the critical points of $F_{\calA(G;m)}$ with the computation of the intersection product $\omega_{G,0,m}$:
\begin{thm}\label{thm:IntersectExplicitHypersurfaces}
    For generic $\boldu\in\C^{\calA(G;m)}$, the hypersurfaces $L_{v,\boldu}\subseteq\Mbar_{0,V\sqcup M}$ satisfy the following properties:
    \begin{enumerate}
        \item $L_{v,\boldu}$ is a representative of the cohomology class $\pi_{N[v]\cup M}^*(\psi_v)$,\label{it:HypersurfaceRepresentsClass}
        \item Under the isomorphism \eqref{eq:IdentifyCVWithFiber}, the intersection $$L_{v,\boldu}\cap(\pi_M^{-1}(y)\cap \M_{0,V\sqcup M})\subseteq\pi_M^{-1}(y)\cap \M_{0,V\sqcup M}$$ is identified with the vanishing locus $$Z\left(\frac{\partial \log F_{\calA(G;m), \boldu}}{\partial z_v}\right)\subseteq\C^{V}\setminus\calA_{\C}(K_V;m),$$ and \label{it:HypersurfaceLogDerivative}
        \item The intersection $\pi_M^{-1}(y)\cap\bigcap_{v\in V}L_{v,\boldu}$ is contained in $\M_{0,V\sqcup M}\subseteq\Mbar_{0,V\sqcup M}$.\label{it:NoIntersectionAtBoundary}
    \end{enumerate}
\end{thm}

As we will see below, statements \eqref{it:HypersurfaceRepresentsClass} and \eqref{it:HypersurfaceLogDerivative} are straightforward to prove, while proving \eqref{it:NoIntersectionAtBoundary}, i.e. showing that there are no intersection points in the boundary of $\Mbar_{0,V\sqcup M}$, is a somewhat involved computational argument. In preparation, we now introduce some notation. 
\begin{notation}\label{not:alphav}
    Let $\Gamma$ be a stable $(V\sqcup M)$-marked genus-zero weighted graph. We use the following notation for the rest of this section --- see Figure \ref{fig:PsiRestrict2} for an illustration.
    \begin{itemize}
        \item As usual, $\Mbar_\Gamma\subseteq\Mbar_{0,V\sqcup M}$ denotes the corresponding boundary stratum, with interior $\M_\Gamma.$ We have $\M_{\Gamma}\cong\prod_{\alpha\in V(\Gamma)}\M_{0,\HE(\Gamma,\alpha)},$ and we write $\pr_\alpha:\M_{\Gamma}\to\M_{0,\HE(\Gamma,\alpha)}$ for the $\alpha$th projection.
        \item For $\alpha\in V(\Gamma)$ and $p\in V\sqcup M$, we write $h_{\alpha\to p}\in\HE(\Gamma,\alpha)$ for the half-edge at $\alpha$ in the direction of $p$. (This is well-defined since $\Gamma$ is a tree.)
        \item Let $v\in V$. For brevity, we write
        \begin{align*}
            \alpha_v&:=\alpha_{N[v]\sqcup M,v}\in V(\Gamma),\\q_v&:=q_{N[v]\sqcup M,v}\in\HE(\Gamma,\alpha_v),\quad\quad\text{and}\\Q_v&:=Q_{N[v]\sqcup M,v}\subseteq\HE(\Gamma,\alpha_v),
        \end{align*}
        where $\alpha_{N[v]\sqcup M,v}$, $q_{N[v]\sqcup M,v}$, and $Q_{N[v]\sqcup M,v}$ are as in Fact \ref{fact:PsiRestrict}.
        \item For $\alpha\in V(\Gamma)$, we write $V_\alpha=\{v\in V:\alpha_v=\alpha\}$.
        \item We name a distinguished element $p_\infty\in M$. For brevity, we write $q_{v,\infty}:=h_{\alpha_v\to p_\infty}\in Q_v$.
    \end{itemize}
\end{notation}
\begin{figure}
        \centering
        \begin{tikzpicture}[scale=1.5]
            \draw[very thick] (-1,0) -- (2,0);
            \draw[very thick] (1,0)--(1,-1);
            \draw[very thick] (1,-1)--++(-30:1);
            \draw (-1,0) node {$\bullet$};
            \draw (1,0) node {$\bullet$};
            \draw (2,0) node {$\bullet$};
            \draw (1,-1) node {$\bullet$};
            \draw (1,-1)++(-30:1) node {$\bullet$};
            \draw[very thick,densely dashed] (1,-1)++(-30:1)--++(0:.6);
            \draw (1,-1)++(-30:1)--++(-60:.6);
            \foreach \th in {-60,-30,0,30,60} {
            \draw[very thick,densely dashed] (-1,0)--++(180+\th:.6);
            };
            \foreach \th in {20,60} {
            \draw[very thick,densely dashed] (2,0)--++(\th:.6);
            };
            \draw[very thick,densely dashed] (1,-1)--++(-90:.6);
            \draw[very thick,densely dashed] (1,-1)--++(-90+20:.6);
            \draw[red,very thick] (1,-1)--++(-90+60:.5);
            \draw[red,very thick] (1,-1)--++(90:.5);
            \draw (-1,0)++(240:.6) node[below] {$v$};
            \draw (2.3,-2) node[below] {$p_{\infty}$};
            \draw[line width=3 pt] (-1,0)++(240:.6)--(-1,0);
            \draw[very thick,red] (1,-1)--++(-90-40:.6);
            \draw[very thick] (1,-1)++(-90-40:.6)--++(-90-40:.6);
            \draw[line width=3 pt] (1,-1)++(-90-40:1.2)--++(-180:.6);
            \draw[line width=3 pt] (1,-1)++(-90-40:1.2)--++(-130:.6);
            \draw[line width=3 pt,red] (1,-1)--++(-90-20:.6);
            \draw[line width=3 pt,red] (1,-1)--++(-90+40:.6);
            \draw[line width=3 pt] (1,-1)++(-90+60:1)--++(-60:.6);
            \draw (.4,-1)++(-90-40:1.2) node[left] {$a$};
            \draw (1,-1)++(-90-40:1.2)++(-130:.6) node[below left] {$b$};
            \draw (2.4,-.55) node {$\alpha_{v}$};
            \draw (2.2,-1.05) node[right] {$q_{v,\infty}=h_{\alpha_v\to p_\infty}$};
            \draw[->] (2.2,-.5) to[out=170,in=45] (1.2,-.9);
            \draw[->] (2.2,-1) to[out=170,in=45] (1.35,-1.15);
            \draw (.4,-.75) node[left] {$q_v=h_{\alpha_v\to v}$};
            \draw[->] (-.2,-.6) to[out=45,in=180] (.95,-.7);
            \draw (-.4,-1.5) node[left] {$h_{\alpha_v\to a}=h_{\alpha_v\to b}$};
            \draw[->] (-.8,-1.4) to[out=10,in=135] (.7,-1.3);
            \filldraw[fill=white,very thick] (-1,0) circle(.17);
            \filldraw[fill=white,very thick] (1,0) circle(.17);
            \filldraw[fill=white,very thick] (2,0) circle(.17);
            \filldraw[fill=white,very thick] (1,-1) circle(.17);
            \filldraw[fill=white,very thick] (1,-1)++(-30:1) circle(.17);

            \filldraw[fill=white,very thick] (1,-1)++(-130:1.2) circle(.17);
            \draw (-1,0) node {$0$};
            \draw (1,0) node {$0$};
            \draw (2,0) node {$0$};
            \draw (1,-1) node {$0$};
            \draw (1,-1)++(-30:1) node {$0$};
            \draw (1,-1)++(-130:1.2) node {$0$};
        \end{tikzpicture}
        
        \caption{An illustration of Notation \ref{not:alphav}. Here $N[v]\sqcup M$ is the set of solid marked half-edges, and $Q_{v}$ is the set of half-edges in red.}
        \label{fig:PsiRestrict2}
    \end{figure}
    
The following lemma describes how $L_{v,\boldu}$ intersects locally closed boundary strata $\M_{\Gamma}$.
\begin{lemma}
    \label{lem:kapranov_morphism_restriction_gen_stratum}
    Let $\Gamma$ be a stable $(V\sqcup M)$-marked genus-0 weighted graph.
    \begin{enumerate}
        \item For generic $\boldu\in\C^{\calA(G;m)},$ $L_{v,\boldu}$ does not contain $\Mbar_{\Gamma}$.\label{it:LvuDoesntContainStratum}
        \item 
        For $q\in Q_v\setminus\{q_v,q_{v,\infty}\}$, let $$a_q = \sum_{\substack{p \in N[v]\sqcup M \\ h_{\alpha_v\to p} = q}} u_{v,p}.$$ Let $l_{v,\boldu,\Gamma}\subseteq\P^{\abs{Q_v}-3}$ be the hyperplane defined by the equation
        \begin{align}
        \label{eqn:kapranov_morphism_restriction_gen_stratum}
            \sum_{q \in Q_v \setminus\{q_v, q_{v,\infty}\}} a_q x_q = 0.
        \end{align} 
        Then $$L_{v,\boldu}\cap\M_\Gamma=(\abs{\psi_{q_v}}\circ\pi_{Q_v}\circ\pr_{\alpha_v})^{-1}(l_{v,\boldu,\Gamma})\subseteq\M_{\Gamma}.$$
        \label{it:EquationOfRestrictionOfLvu} 
    \end{enumerate}
\end{lemma}
\begin{proof}
    For generic $\boldu$, we have that $l_{v,\boldu,\Gamma}$ is really a hyperplane (at least one $a_q$ doesn't vanish). Together with the fact that the composition $\abs{\psi_{q_v}}\circ\pi_{Q_v}\circ\pr_{\alpha_v} \colon \M_\Gamma \rightarrow \P^{\abs{Q_v}-3}$ is dominant, we see that \eqref{it:LvuDoesntContainStratum} follows from 
    \eqref{it:EquationOfRestrictionOfLvu}.

    Note that each element of $(N[v]\sqcup M) \setminus v$ corresponds to a half-edge in $Q_v \setminus \set{q_v}$, where $p_\infty$ corresponds to $q_{v,\infty}$. This induces a  natural inclusion $F \colon\P^{\abs{Q_v} - 3}\into\P^{\abs{N[v]\sqcup M}-3},$ identifying $\P^{\abs{Q_v} - 3}$ with the subspace of $\P^{\abs{N[v]\sqcup M}-3}$ where coordinates corresponding to the same half-edge are equal (and the coordinate of $p_\infty$ is set to 0). Explicitly:
    \begin{align*}
        F \colon \P^{\abs{Q_v} - 3}&\into\P^{\abs{N[v]\sqcup M}-3}\\
        [x_q]_{q\in Q_v\setminus\{q_v,q_{v,\infty}\}}&\mapsto\left[x_{h_{\alpha_v\to w}}\right]_{w\in N[v]\sqcup M \setminus \{v, p_\infty\}},
    \end{align*}
    where we take $x_{h_{\alpha_v\to w}}=0$ if $h_{\alpha_v\to w}=q_{v,\infty}.$ (Note that for $w\in N[v]\sqcup M \setminus \{v, p_\infty\}$, we have $h_{\alpha_v\to w}\ne q_v$, so this map is well-defined.) By definition of $l_{v,\boldu,\Gamma}$, we have: $$F^{-1}(l_{v,\boldu}) = l_{v,\boldu,\Gamma}.$$
    

    Moreover, by applying  Lemma \ref{lem:kapranov_morphism} we see that
    \begin{align*}
        \abs{\psi_v} \circ \pi_{N[v]\sqcup M}(C) &= (F \circ \abs{\psi_{q_v}}\circ\pi_{Q_v}\circ\pr_{\alpha_v})(C)
    \end{align*}
    holds for $C \in \M_{\Gamma}$.
    With $L_{v,\boldu}=(\abs{\psi_v}\circ\pi_{N[v]\cup M})^{-1}(l_{v,\boldu})$, the statement now follows.
\end{proof}
\begin{rem}
    Lemma \ref{lem:kapranov_morphism_restriction_gen_stratum} is a manifestation, via explicit effective divisors, of the equation \eqref{eq:PsiRestrict2} which holds in cohomology.
\end{rem}

We will use the following combinatorial lemma as the basis of our genericity conditions on $\boldu$.
\begin{lemma}
    \label{lem:combinatorial_assignment}
    Let $\Gamma$ be a stable $(V\sqcup M)$-marked genus-0 weighted graph, such that there exists a (necessarily unique) vertex $\alpha_M$ separating all elements of $M$. 
    Suppose there exists $\alpha\in V(\Gamma)$ such that one of the following two holds:
    \begin{enumerate}[label=(\roman*)]
        \item $\alpha=\alpha_M$ and $\abs{V_{\alpha}}>\val(\alpha)-m$, or\label{it:AlphaM}
        \item $\alpha\ne\alpha_M$ and $\abs{V_\alpha}>\val(\alpha)-3$.\label{it:NotAlphaM}
    \end{enumerate}
    Then there is a finite sequence $W=(v_1, \dots, v_k),$ with $ v_i\in V_\alpha$ all distinct, and a map $\phi \colon \{v_1,\ldots,v_k\} \rightarrow V \sqcup M$, such that the following hold:
    \begin{enumerate}
        \item \label{it:phi_neighbor}
        For $i=1,\ldots,k$, we have $\phi(v_i) \in (N[v_i] \sqcup M)$,
        \item \label{it:phi_q_distinct}
        For $i=1,\ldots,k$, we have $h_{\alpha\to\phi(v_i)}\ne h_{\alpha\to p_\infty}$,
        \item \label{it:phi_sorted}
        If $i>j,$ then $\phi(v_i) \ne v_j$, and
        \item \label{it:phi_overdet}
        Let $Q_W = \bigcup_{i=1}^k Q_{ v_i}$. Then $$k > \begin{cases}
            |Q_W| - m&\text{if $\alpha = \alpha_M$}\\
            |Q_W| - 3&\text{if $\alpha \not= \alpha_M$}.
        \end{cases}
        $$
    \end{enumerate}
\end{lemma}
\begin{proof}
    Fix $\alpha$ as in the statement of the Lemma.

Note that by the discussion in Fact \ref{fact:PsiRestrict}, the half-edges $(q_v)_{v\in V_\alpha}$ and $h_{\alpha\to p}$ for $p\in M$ are all distinct.

    \smallskip
    
    \noindent\textit{Case \ref{it:AlphaM}.} Suppose $\alpha=\alpha_M$. In this case $h_{\alpha\to a} \not= h_{\alpha\to b}$ for $a \not= b \in M$. Let $W$ be any ordering of $V_\alpha$, and for each $v\in V_\alpha,$ let $\phi(v)$ be any element of $M$ such that $h_{\alpha\to\phi(v)} \not\in\{q_v,h_{\alpha\to\infty}\}$. (This is possible since $\{h_{\alpha\to a}:a\in M\}$ has at least 3 elements.) Then \eqref{it:phi_neighbor}, \eqref{it:phi_q_distinct}, and \eqref{it:phi_sorted} hold immediately. Note that $\val(\alpha)\ge\abs{Q_W}$ by definition. Thus $$k=\abs{V_\alpha}>\val(\alpha)-m\ge\abs{Q_W}-m,$$ so \eqref{it:phi_overdet} holds.
    \smallskip

\noindent\textit{Case \ref{it:NotAlphaM}.} Suppose $\alpha\ne\alpha_M.$ Let $G'$ be the graph obtained from $G$ by deleting all vertices $w\in V$ with $h_{\alpha\to w}=h_{\alpha\to\infty}$.  (Note that $V_\alpha\subseteq V(G').$)

    Observe that disconnecting $\Gamma$ at $\alpha$ induces a partition of the set $M$. We treat separately the following three cases:
    \begin{enumerate}[label=\textit{Subcase (\alph*):},leftmargin=2\parindent,align=left]
        \item Every connected component of $G'$ contains a vertex not in $V_\alpha$.
        \item There exists a connected component of $G'$ all of whose vertices are in $V_\alpha$, and $\alpha$ induces the trivial partition of $M$ (with one part).
        \item There exists a connected component of $G'$ all of whose vertices are in $V_\alpha$, and $\alpha$ induces a nontrivial partition $M$.
    \end{enumerate}

    \medskip
    
    \noindent\textit{Subcase (a).} Suppose every connected component of $G'$ contains a vertex not in $V_\alpha.$ Let $T$ be a spanning forest of $G'$. For each connected component $C$ of $G'$, we choose a root of $T$ not in $V_\alpha$. This induces an orientation on $T$, with roots as sinks, as well as a partial order on the vertices of $G$, with the roots as the maximal elements. Let $W$ be any ordering of $V_\alpha$ that extends this partial order. For $w\in V_\alpha$, define $\phi(w)$ to be the unique out-neighbor of $w$ in $T$ with the above orientation, which is well-defined since $w$ is not a root of $T$. Then \eqref{it:phi_neighbor}, \eqref{it:phi_q_distinct}, and \eqref{it:phi_sorted} all hold by construction. Again,  $\val(\alpha)\ge\abs{Q_W}$ by definition. Thus $$k=\abs{V_\alpha}>\val(\alpha)-3\ge\abs{Q_W}-3,$$ so \eqref{it:phi_overdet} holds.

    \medskip

    \noindent\textit{Subcase (b).} Suppose there exists a connected component $C$ of $G'$ whose vertices are all in $V_\alpha,$ and disconnecting $\Gamma$ at $\alpha$ induces the trivial partition of $M$, i.e. $h_{\alpha,m_i}=h_{\alpha,p_\infty}$ for all $m_i\in M$. Choose a spanning tree $T$ of $C$ and a root vertex $r$ in $C$, which induces an orientation on $T$ and partial order on $V(C)$ as above. Let $W$ be any ordering of $V(C)\setminus\{r\}$ extending the above partial order. For $w\in V(C)\setminus\{r\}$, define $\phi(w)\in V(C)$ to be the unique out-neighbor of $w$ in the orientation on $T$. Then \eqref{it:phi_neighbor}, \eqref{it:phi_q_distinct}, and \eqref{it:phi_sorted} all hold by construction, while \eqref{it:phi_overdet} follows from 
    \begin{align*}
        \abs{Q_W}=\abs{\{q_v:v\in V(C)\}\cup\{h_{\alpha\to\infty}\}}\leq (\abs{V(C)}-1)+2&=k+2.
    \end{align*}
    (The first equality is by definition of $Q_W$, and the fact that $N[v_i]\subseteq V(C)$.)

    \medskip

    \noindent\textit{Subcase (c).} Suppose there exists a connected component $C$ of $G'$ whose vertices are all in $V_\alpha,$ and disconnecting $\Gamma$ at $\alpha$ induces a nontrivial partition of $M$. The existence of the vertex $\alpha_M\ne\alpha$ separating all elements of $M$ implies that the induced partition of $M$ has exactly two parts $M=M_1\sqcup M_2$, one of which has a single element. Assume without loss of generality $p_\infty\in M_2$. Let $m_i\in M_1.$
    
    Choose a spanning tree $T$ of $C$ and a root vertex $r$ in $C$, which induces an orientation on $T$ and partial order on $V(C)$ as above. Let $W$ be any ordering of $V(C)\subseteq V_\alpha$ extending this partial order. For $w\in V(C)\setminus\{r\}$, define $\phi(w)\in V(C)$ to be the unique out-neighbor of $w$ in the orientation on $T$, and define $\phi(r)=m_i$. Then \eqref{it:phi_neighbor}, \eqref{it:phi_q_distinct}, and \eqref{it:phi_sorted} all hold by construction, while \eqref{it:phi_overdet} follows from the observation $$\abs{Q_W}=\abs{\{q_v:v\in V(C)\}\cup\{h_{\alpha\to m_i},h_{\alpha\to\infty}\}} \leq \abs{V(C)}+2=k+2.$$ 
    (Again, the first equality is by definition of $Q_W$, and the fact that $N[v_i]\subseteq V(C)$.)
\end{proof}

\begin{proof}[Proof of Theorem \ref{thm:IntersectExplicitHypersurfaces}]
    We have $[L_{v,\boldu}]=(\abs{\psi_v}\circ\pi_{N[v]\cup M})^*([l_{v,\boldu}])$ by the definition of the pullback of a divisor. On the other hand, $$(\abs{\psi_v}\circ\pi_{N[v]\cup M})^*([l_{v,\boldu}])=\pi_{N[v]\cup M}^*(\abs{\psi_v}^*([l_{v,\boldu}]))=\pi_{N[v]\cup M}^*\psi_v,$$ so \eqref{it:HypersurfaceRepresentsClass} holds.

    Let $\iota:\C^V\setminus\calA_{\C}(K_V;m)\to\Mbar_{0,V\sqcup M}$ be the inclusion induced by \eqref{eq:IdentifyCVWithFiber}. The pullback of the equation \eqref{eq:HyperplaneEquation} along the composition $\abs{\psi_v}\circ\pi_{N[v]\sqcup M}\circ\iota$ is precisely the logarithmic derivative \eqref{eqn:log_diff}, proving \eqref{it:HypersurfaceLogDerivative}.

It remains to show \eqref{it:NoIntersectionAtBoundary}. It is sufficient to  show that for generic $\boldu\in\C^{\calA(G;m)},$ 
    \begin{align}\label{eq:EmptyIntersectionGoal}
        \M_\Gamma\cap\pi_M^{-1}(y)\cap\bigcap_{v\in V}L_{v,\boldu}=\emptyset
    \end{align} for every stable $(V\sqcup M)$-marked genus-zero weighted graph $\Gamma$ with $\abs{V(\Gamma)}\ge2,$ where $\M_\Gamma \subset \Mbar_{0,V \sqcup M}$ is the corresponding locally closed boundary stratum. (This is because these strata cover $\Mbar_{0,V \sqcup M}\setminus\M_{0,V\sqcup M}$.) Let $\Gamma$ be a stable $(V\sqcup M)$-marked genus-zero weighted graph with $\abs{V(\Gamma)}\ge2$.

    First, observe that if there is no vertex of $\Gamma$ that separates all elements of $M$, then $\pi_M^{-1}(y)\cap\M_\Gamma$ is already empty, since in this case $\pi_M(\M_\Gamma)\subseteq\Mbar_{0,M}\setminus\M_{0,M}$, while $y\in\M_{0,M}$ by definition. Therefore we assume there exists $\alpha_M\in V(\Gamma)$ that separates all elements of $M$.

    Note that $V=\bigsqcup_{\alpha\in V(\Gamma)}V_\alpha$. Note also that since $\Gamma$ is a $(V\sqcup M)$-marked tree, we have $$\sum_{\alpha\in V(\Gamma)}\val(\alpha)=2(\abs{V(\Gamma)}-1)+\abs{V}+m.$$ We therefore compute:
    \begin{align*}
        (\abs{V_{\alpha_M}}-\val(\alpha_M)+m)+\sum_{\alpha \in V(\Gamma)\setminus\{\alpha_M\}} (\abs{V_{\alpha}}-\val(\alpha)+3) &= \abs{V}-\left(\sum_{\alpha\in V(\alpha)}\val(\alpha)\right)+m+3(\abs{V(\Gamma)}-1)\\
        &=-2(\abs{V(\Gamma)}-1)+3(\abs{V(\Gamma)}-1)\\
        &=\abs{V(\Gamma)}-1\\
        &>0.
    \end{align*}
    Thus there exists $\alpha\in V(\Gamma)$ such that $$\abs{V_{\alpha_M}}>\begin{cases}
        \val(\alpha)-m&\text{if } \alpha=\alpha_M,\\
        \val(\alpha_M)-3&\text{if } \alpha\ne\alpha_M.
    \end{cases}$$
    Fix such an $\alpha.$ By Lemma~\ref{lem:combinatorial_assignment}, there exists a finite sequence $W=(v_1,\ldots,v_k)$ and a function $\phi:\{v_1,\ldots,v_k\}\to V\sqcup M$ satisfying conditions \eqref{it:phi_neighbor}-\eqref{it:phi_overdet} of Lemma~\ref{lem:combinatorial_assignment}. We treat the cases $\alpha=\alpha_M$ and $\alpha\ne\alpha_M$ separately.

\medskip

    \noindent\textit{Case \ref{it:AlphaM}.} Suppose $\alpha=\alpha_M.$ Then we have a natural injection $M\into Q_W$ by $p\mapsto h_{\alpha\to p}$, giving a forgetful map $\pi:\M_{0,Q_W}\to\Mbar_{0,M}.$ We therefore have a commutative diagram
    \begin{align*}
        \begin{tikzcd}[ampersand replacement=\&]
            \M_\Gamma\arrow[d,"\pr_\alpha"]\arrow[r,"\iota_\Gamma",hook]\&\Mbar_{0,V\sqcup M}\arrow[dd,"\pi_M"]\\
            \M_{0,\HE(\Gamma,\alpha)}\arrow[d,"\pi_{Q_W}"]\&\\
            \M_{0,Q_W}\arrow[r,"\pi"]\&\Mbar_{0,M}
        \end{tikzcd}
    \end{align*}
    In particular, $\pi_M^{-1}(y)\cap\M_\Gamma\subseteq\M_\Gamma$ can be written as a preimage under the map $\pi_{Q_W}\circ\pr_{\alpha}:\M_\Gamma\to\M_{0,Q_W}$: specifically, $$\pi_M^{-1}(y)\cap\M_\Gamma=(\pi_{Q_W}\circ\pr_{\alpha})^{-1}(\pi^{-1}(y)).$$ By Lemma \ref{lem:kapranov_morphism_restriction_gen_stratum}\eqref{it:EquationOfRestrictionOfLvu}, for all $i=1,\ldots,k$, $L_{v_i,\boldu}\cap\M_\Gamma$ can also be written as a preimage under this same map $\pi_{Q_W}\circ\pr_{\alpha}:\M_\Gamma\to\M_{0,Q_W}$:  specifically, $$L_{v_i,\boldu}\cap\M_\Gamma=(\pi_{Q_W}\circ\pr_{\alpha})^{-1}((\abs{\psi_{q_{v_i}}}\circ\pi_{Q_{v_i}})^{-1}(l_{{v_i},\boldu,\Gamma})).$$
    Therefore, to prove \eqref{eq:EmptyIntersectionGoal} in this case, it is sufficient to show that for generic $\boldu\in\C^{\calA(G;m)},$ the intersection \begin{align}\label{eq:AlphaMIntersectionOnQW}
        \pi^{-1}(y)\cap\bigcap_{i=1}^{k} (\abs{\psi_{q_{v_i}}}\circ\pi_{Q_{v_i}})^{-1}(l_{{v_i},\boldu,\Gamma})\subseteq\M_{0,Q_W}
    \end{align} is empty. We will prove this inductively; specifically, we will show that for $0\le s\le k$, and for general $\boldu$, the intersection $\pi^{-1}(y)\cap\bigcap_{i=1}^s(\abs{\psi_{q_{v_i}}}\circ\pi_{Q_{v_i}})^{-1}(l_{{v_i},\boldu,\Gamma})$ has codimension at least $m-3+s$ in $\M_{0,Q_W}.$ Since $k>\abs{Q_W}-m$ by condition \eqref{it:phi_overdet} of Lemma \ref{lem:combinatorial_assignment}, this will imply that the intersection \eqref{eq:AlphaMIntersectionOnQW} has codimension at least $$m-3+k>m-3+\abs{Q_W}-m=\abs{Q_W}-3=\dim(\M_{0,Q_W})$$ in $\M_{0,Q_W}$, hence is empty.
    
    The base case $s=0$ is immediate from flatness of $\pi$, see Notation \ref{not:BasicMgnNotation}. For the inductive step, suppose there exists a dense open subset $U\subseteq\C^{\calA(G;m)}$ such that for $\boldu\in U,$ the intersection \begin{align}\label{eq:IntersectionUpToS}
        \pi^{-1}(y)\cap\bigcap_{i=1}^s(\abs{\psi_{q_{v_i}}}\circ\pi_{Q_{v_i}})^{-1}(l_{{v_i},\boldu,\Gamma})\subseteq\M_{0,Q_W}
    \end{align} has codimension at least $m-3+s$. We may equivalently state this in terms of the ``universal'' intersection over $U$. That is, define $$l_{{v},U,\Gamma}=\{(x,\boldu):\boldu\in U,x\in l_{v,u,\Gamma}\}\subset\P^{\abs{Q_{v}}-3}\times U$$ and let 
    \begin{align*}
        Z_s:&=\{(x,\boldu):\boldu\in U,x\in\pi^{-1}(y)\cap\bigcap_{i=1}^s(\abs{\psi_{q_{v_i}}}\circ\pi_{Q_{v_i}})^{-1}(l_{{v_i},\boldu,\Gamma})\}\\
        &=(\pi^{-1}(y)\times U)\cap\bigcap_{i=1}^s(\abs{\psi_{q_{v_i}}}\circ\pi_{Q_{v_i}})^{-1}(l_{{v_i},U,\Gamma})\subseteq\M_{0,Q_W}\times U.
    \end{align*} 
    By construction of $U$, $Z_s$ has codimension at least $m-3+s$ in $\M_{0,Q_W}\times U.$ Note that $l_{{v},U,\Gamma}\subset\P^{\abs{Q_v}-3}\times U$ is still cut out by the equation \eqref{eqn:kapranov_morphism_restriction_gen_stratum}, where the variables $u_{v,p}$ are considered as functions on $U$. 
    
    By Lemma \ref{lem:combinatorial_assignment} condition \eqref{it:phi_sorted}, we have $\phi(v_{s+1})\not\in\{v_1,\ldots,v_s\}$, and therefore the variable $u_{v_{s+1},\phi(v_{s+1})}$ does not appear in any of the equations \eqref{eqn:kapranov_morphism_restriction_gen_stratum} for the hyperplanes $l_{v_i,U,\Gamma}$ for $i=1,\ldots,s$. That is, $Z_s$ is the preimage of a subvariety $\underline{Z_s}\subseteq\M_{0,Q_W}\times\C^{\calA(G;m)\setminus\{(v_{s+1},\phi(v_{s+1}))\}}$, of codimension at least $m-3+s,$ under the map $$\M_{0,Q_W}\times U\to\M_{0,Q_W}\times\C^{\calA(G;m)\setminus\{(v_{s+1},\phi(v_{s+1}))\}}$$ that projects along the $u_{v_{s+1},\phi(v_{s+1})}$-direction.
    
    On the other hand, by \eqref{eqn:kapranov_morphism_restriction_gen_stratum}, $l_{v_{s+1},U,\Gamma}$ is the graph of the rational function $$g:\P^{\abs{Q_{v_{s+1}}}-3}\times\C^{\calA(G;m)\setminus\{(v_{s+1},\phi(v_{s+1}))\}}\dashrightarrow\A^1$$ defined by $$g([x_q],(u_{H}))=-\frac{1}{x_{h_{\alpha\to\phi(v_{s+1})}}}\sum_{\substack{q\in Q_v\setminus\{q_{v_{s+1}},h_{\alpha\to p_\infty}\}}}x_q\sum_{\substack{p\in N[v_{s+1}]\sqcup M\\h_{\alpha\to p=q}\\p\ne\phi(v_{s+1})}}u_{v_{s+1},p}.$$ (Note that the codomain $\A^1$ of $g$ is identified with the coordinate direction corresponding to the hyperplane $(v_{s+1},\phi(v_{s+1}))$.) Note that in writing $x_{h_{\alpha\to\phi(v_{s+1})}}$, we have used $h_{\alpha\to\phi(v_{s+1})}\not\in\{q_{v_{s+1}},h_{\alpha\to p_{\infty}}\}.$ This follows from conditions \eqref{it:phi_neighbor} and \eqref{it:phi_q_distinct} of Lemma \ref{lem:combinatorial_assignment}, and the fact that $v_{s+1}\in V_\alpha$.
    
    Since $l_{v_{s+1},U,\Gamma}$ is the graph of a rational function on $\P^{\abs{Q_{v_{s+1}}}-3}\times\C^{\calA(G;m)\setminus\{(v_{s+1},\phi(v_{s+1}))\}},$ and $$\abs{\psi_{q_{v_{s+1}}}}\circ\pi_{Q_{v_{s+1}}}:\M_{0,Q_W}\to\P^{\abs{Q_{v_{s+1}}}-3}$$ is dominant, we have that $(\abs{\psi_{q_{v_i}}}\circ\pi_{Q_{v_i}})^{-1}(l_{v_{s+1},U,\Gamma})$ is the graph of a rational function $$\tilde g:\M_{0,Q_W}\times\C^{\calA(G;m)\setminus\{(v_{s+1},\phi(v_{s+1}))\}}\dashrightarrow\A^1.$$ Thus $Z_{s+1}=Z_s\cap(\abs{\psi_{q_{v_i}}}\circ\pi_{Q_{v_i}})^{-1}(l_{v_{s+1},U,\Gamma})$ is the graph of the rational function $\tilde g|_{\underline{Z_s}}$. Hence we have a natural inclusion $Z_{s+1}\into\underline{Z_{s}}$. Thus $\dim(Z_{s+1})\le\dim(\underline{Z_s})=\dim(Z_s)-1,$ so the codimension of $Z_{s+1}\subseteq\M_{0,Q_W}\times U$ is at least $m-3+s+1$. Thus for generic $\boldu\in U,$ \eqref{eq:IntersectionUpToS} has codimension at least $m-3+s+1$ in $\M_{0,Q_W}$.

    Since $k>\abs{Q_W}-m$, we conclude by induction that \eqref{eq:AlphaMIntersectionOnQW} is empty, and hence that \eqref{eq:EmptyIntersectionGoal} is empty. Since this holds for all $\Gamma,$ we conclude part \eqref{it:NoIntersectionAtBoundary} of the Theorem.

\medskip

    \noindent\textit{Case \ref{it:NotAlphaM}} The argument when $\alpha\ne\alpha_M$ is essentially identical --- one simply removes $\pi^{-1}(y)$ from the intersection \eqref{eq:AlphaMIntersectionOnQW} and replaces the codimension $m-3+s$ with $s$ throughout.
    \end{proof}

Finally, we show that Theorem \ref{thm:IntersectExplicitHypersurfaces} implies Theorem \ref{thm:critical_point_interpretation}.

\medskip

\noindent\textbf{Theorem \ref{thm:critical_point_interpretation}.}
    For general $\boldu \in \C^{\# \calA(G;m)}$, $F_{\calA(G;m), \boldu}$ has exactly $\omega_{G,0,m}$ critical points.

\begin{proof}
    Let $\boldu \in \C^{\#\calA(G;m)}$ be general, and let $L_{v,\boldu}$ be as in Theorem \ref{thm:IntersectExplicitHypersurfaces}. By Lemma \ref{lem:CritAvoidsKV}, the critical locus $\Crit(F_{\calA(G;m),\boldu})$ is contained in $\C^V\setminus\calA_{\C}(K_V;m)$. By definition, we have $$\Crit(F_{\calA(G;m),\boldu})=\bigcap_{v\in V}Z\left(\frac{\partial \log F_{\calA(G;m), \boldu}}{\partial z_v}\right)\subseteq\C^V\setminus\calA_{\C}(K_V;m).$$ By Theorem \ref{thm:IntersectExplicitHypersurfaces}\eqref{it:HypersurfaceLogDerivative}, the isomorphism \eqref{eq:IdentifyCVWithFiber} restricts to an isomorphism $$\Crit(F_{\calA(G;m),\boldu})\cong\pi_M^{-1}(y)\cap \M_{0, V\sqcup M}\cap\bigcap_{v\in V}L_{v,\boldu}\subseteq\Mbar_{0,V\sqcup M}.$$ By Theorem \ref{thm:IntersectExplicitHypersurfaces}\eqref{it:NoIntersectionAtBoundary}, the subscheme $\pi_M^{-1}(y)\cap\bigcap_{v\in V}L_{v,\boldu}$ in contained in $\M_{0,V\sqcup M}\subset\Mbar_{0,V\sqcup M}$, i.e. we have
    \begin{align}\label{eq:CritAsIntersectionInMbar}
        \Crit(F_{\calA(G;m),\boldu})\cong\pi_M^{-1}(y)\cap\bigcap_{v\in V}L_{v,\boldu}\subseteq\Mbar_{0,V\sqcup M}.
    \end{align}
    The intersection in \eqref{eq:CritAsIntersectionInMbar} has expected dimension zero and consists of finitely many reduced points by Theorem \ref{thm:varschenko}\eqref{it:FinitelyManyCriticalPoints}. It follows that the intersection in \eqref{eq:CritAsIntersectionInMbar} is transverse, and hence $\Crit(F_{\calA(G;m),\boldu})$ represents the intersection product $[\pi_M^{-1}(y)]\cdot\prod_{v\in V}[L_{v,\boldu}]\in H^*(\Mbar_{0,V\sqcup M})$ as a cycle. That is,
    \begin{align}\label{eq:CritIntersectionNumber}
        \#\Crit(F_{\calA(G;m),\boldu})=\int_{\Mbar_{0,V\sqcup M}}[\pi_M^{-1}(y)]\cdot\prod_{v\in V}[L_{v,\boldu}].
    \end{align}
    
    By Theorem \ref{thm:IntersectExplicitHypersurfaces}\eqref{it:HypersurfaceRepresentsClass}, we have $[L_{v,\boldu}]=\pi_{N[v]\cup M}^*(\psi_v)$. We also have $[\pi_M^{-1}(y)]=\pi_M^*([pt]);$ this follows from the fact that $y$ is a smooth point of $\Mbar_{0,M}$ and the fact that $\pi_M^{-1}(y)$ has the expected dimension. Substituting these into \eqref{eq:CritIntersectionNumber} completes the proof.
\end{proof}

\section{Polynomials associated to directed graphs}\label{sec:Digraphs}
In this section we prove Definition/Theorem \ref{defthm:DigraphPolynomials} and Proposition \ref{prop:SourceRecursion}.

\medskip

\noindent \textbf{Definition/Theorem \ref{defthm:DigraphPolynomials}.}
    There exist monic polynomials $\chi_G^{\mathrm{in}},\chi_G^{\mathrm{out}}\in\Z[x]$ of degree $\abs{V}$ such that, for $g\ge0$ and $m\ge0$ with $2g-2+m>0,$ we have
    \begin{align*}
        \omega_{G,g,m}^{\mathrm{in}}&=(-1)^{\abs{V}}\chi_G^{\mathrm{in}}(-(2g-2+m))\\
        \omega_{G,g,m}^{\mathrm{out}}&=(-1)^{\abs{V}}\chi_G^{\mathrm{out}}(-(2g-2+m)).\nonumber
    \end{align*}
    Furthermore, for the case $(g,m)=(1,0),$ we have
    \begin{align*}
        \omega_{G,1,0}^{\mathrm{in}}&=(-1)^{\abs{V}-1}\left.\deriv{}{x}\chi_G^{\mathrm{in}}(x)\right|_{x=0}&&\text{and}&
        \omega_{G,1,0}^{\mathrm{out}}&=(-1)^{\abs{V}-1}\left.\deriv{}{x}\chi_G^{\mathrm{out}}(x)\right|_{x=0}.
    \end{align*}

\medskip

\begin{proof}
    Since $\omega^{\mathrm{in}}_{G,g,m}$ and $\omega^{\mathrm{out}}_{G,g,m}$ are permuted by reversing the directions of all edges of $G$, by symmetry, it is sufficient to prove the Theorem for $\chi_G^{\mathrm{in}}.$ We do this first in the case $g=0$.
    
    The proof in genus zero is by induction, which works as follows. For a graph $G$ with at least one edge, we will reduce the statement of the Theorem to the statement for graphs with either fewer vertices than $G$ or fewer edges. As before, the bases cases are all of the edgeless directed graphs $G$ --- this case is immediate from (the base case of) Theorem \ref{thm:main}, taking $\chi_G^{\mathrm{in}}(x)=\chi_G^{\mathrm{out}}(x)=\chi_G(x)=x^{\abs{V}}$.
    
    For the inductive step, let $G$ be a directed graph with at least one edge, say $e=u\to w$. Then 
    \begin{align*}
        \omega^{\mathrm{in}}_{G,g,m}-\omega^{\mathrm{in}}_{G\setminus e,g,m}&=\int_{\Mbar_{g,V\sqcup M}}\pi_M^*([pt])\cdot(\pi_{N^{\mathrm{in}}[w]\sqcup M}^*(\psi_w)-\pi_{(N^{\mathrm{in}}[w]\setminus u)\sqcup M}^*(\psi_w))\cdot\prod_{v\in V\setminus w}\pi_{N^{\mathrm{in}}[w]\sqcup M}^*(\psi_v)\nonumber\\
        &=\int_{\Mbar_{g,V\sqcup M}}\pi_M^*([pt])\cdot(\pi_{N^{\mathrm{in}}[w]\sqcup M}^*(D_{uw}))\cdot\prod_{v\in V\setminus w}\pi_{N^{\mathrm{in}}[w]\sqcup M}^*(\psi_v).
    \end{align*}
    By Fact \ref{fact:BoundaryPullback}, we have $$\pi_{N^{\mathrm{in}}[w]\sqcup M}^*(D_{uw})=\sum_{\substack{J\subseteq V\setminus N^{\mathrm{in}}[w]}}\left(\begin{tikzpicture}[baseline=-0.65ex]
            \draw (0,0) -- (1,0);
            \draw (0,0) node {$\bullet$};
            \draw (1,0) node {$\bullet$};
            \draw (0,0)--++(135:.6);
            \draw (135:.6) node[left] {$u$};
            \draw (0,0)--++(180:.6);
            \draw (180:.6) node[left] {$w$};
            \foreach \x in {215,220,225,230,235} {
            \draw (0,0)--++(\x:.6);};
            \draw (225:.6) node[below left] {$J$};
            \foreach \th in {-60,-40,-20,0,20,40,60} {
            \draw (1,0)--++(\th:.6);
            };
            \draw (1.6,0) node[right] {$S\setminus(\{u,w\}\cup J)$};
            \filldraw[fill=white,thick] (0,0) circle(.25);
            \filldraw[fill=white,thick] (1,0) circle(.25);
            \draw (0,0) node {$0$};
            \draw (1,0) node {$g$};
        \end{tikzpicture}\right)$$
    That is, we have $$\omega^{\mathrm{in}}_{G,g,m}-\omega^{\mathrm{in}}_{G\setminus e,g,m}=\sum_{\substack{J\subseteq V\setminus N^{\mathrm{in}}[w]}}I_1(J,m)I_2(J,m),$$ where $I_1(J,m)$ is an integral over $\Mbar_{0,\{u,w\}\cup J \cup *}$, and $I_2(J,m)$ is an integral over $\Mbar_{g,(S\setminus(\{u,w\}\cup J)) \cup *}$. Observe that $I_1(J):=I_1(J,m)$ is actually independent of $m$. By dimension-counting, $I_1(J)$ is nonzero \emph{only} if there is a unique $j\in J\cup\{u\}$ such that $N^{\mathrm{in}}[j]\cap(J\cup\{u, w\})=\{j\}.$ By Fact \ref{fact:PsiRestrict}, we have $I_2(J,m)=\omega_{G'_J,g,m}$, where $G'_J$ is obtained from $G$ by deleting the vertices in $(J\cup\{u,w\})\setminus\{j\}$, and for any remaining vertex $v$, adding an edge $v\to j$ if (no such edge already exists and) there exists an edge from $v$ to an element of $(J\cup\{u,w\})$. Note that $G'_J$ has $\abs{V}-(\abs{J}+1)<\abs{V}$ vertices.

    We have now shown $$\omega^{\mathrm{in}}_{G,g,m}=\omega^{\mathrm{in}}_{G\setminus e,g,m}+\sum_{\substack{J\subseteq V\setminus N^{\mathrm{in}}[w]}}I_1(J)\omega_{G'_J,g,m}.$$
    By the inductive hypothesis, we may write 
    \begin{align*}
        \omega_{G\setminus e,g,m}&=(-1)^{\abs{V}}\chi_{G\setminus e}^{\mathrm{in}}(-(2g-2+m))\\
        \omega_{G'_J,g,m}&=(-1)^{\abs{V}-\abs{J}-1}\chi_{G'_J}^{\mathrm{in}}(-(2g-2+m)),
    \end{align*}
    where $\chi_{G\setminus e}^{\mathrm{in}}$ is a monic polynomial of degree $\abs{V}$ and $\chi_{G'_J}^{\mathrm{in}}$ is a polynomial of degree less than $\abs{V}$. Thus we define a monic polynomial of degree $\abs{V}$: $$\chi_{G}^{\mathrm{in}}:=\chi_{G\setminus e}^{\mathrm{in}}+\sum_{\substack{J\subseteq V\setminus N^{\mathrm{in}}[w]}}(-1)^{-\abs{J}-1}I_1(J)\chi_{G'_J}^{\mathrm{in}},$$ which satisfies
    \begin{align*}
        (-1)^{\abs{V}}\chi_{G}^{\mathrm{in}}(-(2g-2+m))&=(-1)^{\abs{V}}\chi_{G\setminus e}^{\mathrm{in}}(-(2g-2+m))+\sum_{\substack{J\subseteq V\setminus N^{\mathrm{in}}[w]}}(-1)^{\abs{V}-\abs{J}-1}I_1(J)\chi_{G'_J}^{\mathrm{in}}(-(2g-2+m))\\
        &=\omega_{G\setminus e}^{\mathrm{in}}+\sum_{\substack{J\subseteq V\setminus N^{\mathrm{in}}[w]}}I_1(J)\omega_{G'_J,g,m}^{\mathrm{in}}\\
        &=\omega_{G,g,m}^{\mathrm{in}}.
    \end{align*}
    This proves the $g=0$ case of Definition/Theorem \ref{defthm:DigraphPolynomials}.

Now suppose $g>0$ and $m\ge0$ with $(g,m)\ne(1,0)$. By the $g=0$ case, it suffices to show that $$\omega^{\mathrm{in}}_{G,g,m}=\omega^{\mathrm{in}}_{G,0,m+2g}\quad\quad\text{and}\quad\quad\omega^{\mathrm{out}}_{G,g,m}=\omega^{\mathrm{out}}_{G,0,m+2g}.$$ Let $\Gamma$ denote the $V\sqcup M$-marked genus-$g$ weighted stable graph that has a single weight-zero vertex (at which all marks are attached) and $g$ loops, and let $\Gamma'$ denote the $M$-marked graph obtained from $\Gamma$ be forgetting $V$. We have $\Mtilde_{\Gamma}=\Mbar_{0,V\sqcup M\sqcup Q},$ where $\abs{Q}=2g,$ giving the following commutative diagram:
    $$\begin{tikzcd}
        \Mbar_{0,V\sqcup M\sqcup Q}\arrow[r,"\iota_\Gamma"]\arrow[d,"\pi_{M\sqcup Q}"]&\Mbar_{g,V\sqcup M}\arrow[d,"\pi_M"]\\
        \Mbar_{0,M\sqcup Q}\arrow[r,"\iota_{\Gamma'}"]&\Mbar_{g,M}
    \end{tikzcd}$$
    Observe that by Fact \ref{fact:PsiRestrict}, we have $$\Psi_{G,0,m+2g}=\iota_\Gamma^*\Psi_{G,g,m}.$$
    Thus we have
    \begin{align}\label{eq:HigherGenusReduceToGenusZero}
        \omega^{\mathrm{in}}_{G,0,2g+m}&=\int_{\Mbar_{0,V\sqcup M\sqcup Q}}\Psi_{G,0,m+2g}\cdot\pi_{M\sqcup Q}^*([pt])\\
        &=\int_{\Mbar_{g,V\sqcup M}}(\iota_\Gamma)_*(\Psi_{G,0,m+2g}\cdot\pi_{M\sqcup Q}^*([pt]))\nonumber\\
        &=\int_{\Mbar_{g,V\sqcup M}}(\iota_\Gamma)_*(\iota_\Gamma^*\Psi_{G,g,m}\cdot\pi_{M\sqcup Q}^*([pt]))\nonumber\\
        &=\int_{\Mbar_{g,V\sqcup M}}\Psi_{G,g,m}\cdot(\iota_\Gamma)_*\pi_{M\sqcup Q}^*([pt]). \nonumber
    \end{align}
    Let $(C,\vec p)\in\Mbar_{0,M\sqcup Q}$ be a general point, and let $(C',{\vec p\,}')=\iota_{\Gamma'}(C,\vec p)$. Since $\pi_{M\sqcup Q}$ has reduced fibers (see the proof of Fact \ref{fact:BoundaryPullback}), $\pi_{M\sqcup Q}^{-1}(C,\vec p)$ is reduced, and furthermore $\pi_{M\sqcup Q}^{-1}(C,\vec p)$ is an iterated blowup of $C^{\abs{V}}.$ Similarly $\pi_{M}^{-1}(C',{\vec p\,}')$ is a reduced scheme\footnote{Here $\pi_{M}^{-1}(C',{\vec p\,}')$ means the fiber product $\{(C,p)\}\times_{\Mbar_{g,m}}\Mbar_{g,V\sqcup M}$, which is a scheme since $\pi_M$ is representable \cite{Knudsen1983}.}, and is an iterated blowup of $(C')^{\abs{V}}.$ Under these identifications, the restriction of $\iota_\Gamma$ to $\pi_{M\sqcup Q}^{-1}(C,\vec p)$ is, up to birational transformation, identified with the (power of the) gluing map $C^{\abs{V}}\to(C')^{\abs{V}}$. In particular, we have $$\pi_M^*([pt])=\pi_M^*(\iota_{\Gamma'})_*([pt])=\pi_M^*([(C',{\vec p\,}')])=[\pi_{M}^{-1}(C',{\vec p\,}')]=(\iota_{\Gamma})_*([\pi_{M\sqcup Q}^{-1}(C,\vec p)])=(\iota_{\Gamma})_*\pi_{M\sqcup Q}^*([pt]).$$
    Thus by \eqref{eq:HigherGenusReduceToGenusZero}, we have $$\omega^{\mathrm{in}}_{G,0,2g+m}=\int_{\Mbar_{g,V\sqcup M}}\Psi_{G,g,m}\cdot\pi_M^*([pt])=\omega_{G,g,m},$$ as desired. The case $(g,m)=(1,0)$ is essentially identical.
\end{proof}
\begin{rem}\label{rem:HigherGenusFromGenusZero}
    The second half of the above argument (i.e. the case $g>0$) also gives an alternative way to derive the case $g>0$ of Theorem \ref{thm:main} from the $g=0$ case.
\end{rem}
\medskip

\noindent\textbf{Proposition \ref{prop:SourceRecursion}}
    Let $G=(V,E)$ be a finite simple directed graph, let $v\in V$, and let $G\setminus\{v\}$ be the graph obtained by deleting $v$. If $v$ is a sink, then
    \begin{align*}
        \chi_G^{\mathrm{in}}(x)=(x-\indeg(v))\cdot\chi_{G\setminus v}^{\mathrm{in}}(x).
    \end{align*}
    If $v$ is a source, then
    \begin{align*}
        \chi_G^{\mathrm{out}}(x)=(x-\outdeg(v))\cdot\chi_{G\setminus v}^{\mathrm{out}}(x).
    \end{align*}
    Here $\indeg(v)$ and $\outdeg(v)$ denote the number of incoming and outgoing edges at $v$, respectively.

\medskip

\begin{proof}
    As in the proof of Definition/Theorem \ref{defthm:DigraphPolynomials}, it is sufficient to prove the statement for $\chi_{G}^{\mathrm{in}}.$ Let $w\in V$ be a sink. Then every factor in $\pi_M^*([pt])\cdot\prod_{v\in V}\pi_{N^{\mathrm{in}}[v]\sqcup M}^*(\psi_v)$ \emph{except} for $\pi_{N^{\mathrm{in}}[w]\sqcup M}^*(\psi_w)$ is pulled back from $\Mbar_{g,(V\setminus\{w\})\sqcup M}.$ By the projection formula, we have
    \begin{align*}
        \omega_{G,g,m}&=\int_{\Mbar_{g,V\sqcup M}}\Psi_{G,g,m}^{\mathrm{in}}\cdot\pi_M^*([pt])\\
        &=\int_{\Mbar_{g,(V\setminus\{w\})\sqcup M}}(\pi_{(V\setminus\{w\})\sqcup M})_*\pi_{N^{\mathrm{in}}[w]\sqcup M}^*(\psi_w)\cdot\Psi_{G\setminus w,g,m}^{\mathrm{in}}\cdot\pi_M^*([pt])
    \end{align*}
    By Fact \ref{fact:PsiPullback}, we have 
    \begin{align}\label{eq:SinkPsiPullback}
        \pi_{N^{\mathrm{in}}[w]\sqcup M}^*(\psi_w)=\psi_w\quad-\quad\sum_{\substack{J\subseteq S\setminus(N^{\mathrm{in}}[w]\sqcup M)\\J\ne\emptyset}}\quad
        {\begin{tikzpicture}[baseline=-0.65ex]
            \draw (0,0) -- (1,0);
            \draw (0,0)--++(135:.6);
            \draw (135:.6) node[left] {$w$};
            \foreach \x in {215,220,225,230,235} {
            \draw (0,0)--++(\x:.6);};
            \draw (225:.6) node[below left] {$J$};
            \foreach \th in {-60,-40,-20,0,20,40,60} {
            \draw (1,0)--++(\th:.6);
            };
            \draw (1.6,0) node[right] {$S\setminus(J\cup\{w\})$};
            \filldraw[fill=white,thick] (0,0) circle(.25);
            \filldraw[fill=white,thick] (1,0) circle(.25);
            \draw (0,0) node {$0$};
            \draw (1,0) node {$g$};
        \end{tikzpicture}}
    \end{align}
    We now apply $(\pi_{(V\setminus\{w\})\sqcup M})_*$ to the right side of \eqref{eq:SinkPsiPullback}. From the first term we get $$(\pi_{(V\setminus\{w\})\sqcup M})_*\psi_w=(2g-2+(\abs{V}-1)+m)\cdot[\Mbar_{g,(V\setminus\{w\})\sqcup M}],$$ see \cite[Rem. 2.2.1]{KockNotes}. (The factor comes from the degree of the relative cotangent bundle on the universal curve over $\Mbar_{g,(V\setminus\{w\})\sqcup M}$.) The boundary divisors in \eqref{eq:SinkPsiPullback} are contracted in dimension unless $\abs{J}=1,$ in which case $\pi_{(V\setminus\{w\})\sqcup M}$ is an isomorphism on each such boundary divisor. Thus the pushforward via $\pi_{(V\setminus\{w\})\sqcup M}$ of the sum in \eqref{eq:SinkPsiPullback} is $$(\abs{V}-\abs{N^{\mathrm{in}}[w]})\cdot[\Mbar_{g,(V\setminus\{w\})\sqcup M}].$$ We thus have
    \begin{align}\label{eq:SinkPushforwardFormula}
        \omega_{G,g,m}
        &=\int_{\Mbar_{g,(V\setminus\{w\})\sqcup M}}(\pi_{(V\setminus\{w\})\sqcup M})_*\pi_{N^{\mathrm{in}}[w]\sqcup M}^*(\psi_w)\cdot\Psi_{G\setminus w,g,m}^{\mathrm{in}}\cdot\pi_M^*([pt])\\
        &=(2g-3+m+\abs{N^{\mathrm{in}}[w]})\cdot[\Mbar_{g,(V\setminus\{w\})\sqcup M}]\cdot\Psi_{G\setminus w,g,m}^{\mathrm{in}}\cdot\pi_M^*([pt])\nonumber\\
        &=(2g-3+m+\abs{N^{\mathrm{in}}[w]})\cdot\omega_{G\setminus w,g,m}\nonumber\\
        &=(2g-2+m+\indeg(w))\cdot\omega_{G\setminus w,g,m}.\nonumber
    \end{align}
    By Definition/Theorem \ref{defthm:DigraphPolynomials} and \eqref{eq:SinkPushforwardFormula}, we have $$(-1)^{\abs{V}}\chi_G(-(2g-2+m))=(2g-2+m+\indeg(w))\cdot(-1)^{\abs{V}-1}\chi_{G\setminus w}(-(2g-2+m))$$ for all $g$ and $m$. Thus the polynomials $(-1)^{\abs{V}}\chi_G(x)$ and $(-x+\indeg(w))\cdot(-1)^{\abs{V}-1}\chi_G(x)$ agree at every negative integer, hence they coincide, i.e.: 
    \begin{align*}
        \chi_G(x)&=(x-\indeg(w))\chi_{G\setminus w}(x).\qedhere
    \end{align*}
\end{proof}

\bibliography{StableCurvesAndChromaticPolynomialsAccepted.bbl}
\bibliographystyle{amsalpha}

\end{document}